\documentclass[10pt]{article}

\RequirePackage{amsmath}
\RequirePackage{amsfonts}
\RequirePackage{amssymb}
\RequirePackage{hyperref}
\RequirePackage{multirow}
\RequirePackage{algorithm}
\RequirePackage{algpseudocode}
\RequirePackage{graphicx}
\RequirePackage{xcolor}
\RequirePackage{amsthm}
\usepackage{cleveref}
\usepackage{booktabs}
\allowdisplaybreaks

\usepackage{dpr_fec,ifthen,dsfont}

\newtheorem{definition}{Definition}[section]
\newtheorem{theorem}{Theorem}[section]
\newtheorem{assumption}{Assumption}[section]
\newtheorem{corollary}[theorem]{Corollary}
\newtheorem{lemma}[theorem]{Lemma}

\newtheorem{remark}[definition]{Remark}

\usepackage{isetechreport}
\def\reportyear{26T}
\def\reportno{019}
\def\revisionno{0}
\def\originaldate{\today}
\def\revisiondate{\today}
\coraltrue
\cvcrfalse

\begin{document}

\title{A Stochastic Proximal-Gradient Method for Regularized Generally Constrained Continuous Optimization


}

\author{Frank E.~Curtis\thanks{E-mail: frank.e.curtis@lehigh.edu}}
\author{Xiaoyi Qu\thanks{E-mail: xiq322@lehigh.edu}}
\author{Daniel P.~Robinson\thanks{E-mail: daniel.p.robinson@lehigh.edu}}
\affil{Department of Industrial and Systems Engineering, Lehigh University}
\titlepage

\maketitle

\begin{abstract}
  A stochastic proximal-gradient method is proposed and analyzed for solving regularized optimization problems subject to nonlinear equality constraints and nonnegativity bounds. The objective function's values and gradients are available only through inexact stochastic oracles. The method decomposes each trial step into one component that reduces linearized infeasibility and a second component that seeks optimality while preserving that reduction. Trial steps are accepted or rejected through a relaxed sufficient decrease condition subject to stochastic merit function evaluations. Without assuming unbiasedness of stochastic gradient estimates, it is proved under reasonable assumptions that the method reaches an $\epsilon$-KKT point within $\mathcal{O}(\epsilon^{-2})$ iterations with high probability. This guarantee further implies almost-sure termination in finitely many iterations and an $\mathcal{O}(\epsilon^{-2})$ bound on the expected iteration complexity.
\end{abstract}

\newcommand{\tauktrial}{\tau_{k,\text{trial}}}
\newcommand{\tautrialmin}{\tau_{\text{min,trial}}}
\newcommand{\xtrial}{x^{\text{trial}}}
\newcommand{\finf}{f_{\text{inf}}}
\newcommand{\phibar}{\bar\phi}
\newcommand{\alphamin}{\alpha_{\min}}
\newcommand{\taumin}{\tau_{\min}}
\newcommand{\sigmamin}{\sigma_{\min}}
\newcommand{\sigmamax}{\kappa_J}
\newcommand{\lipg}{L_{g}}
\newcommand{\lipJ}{L_{J}}
\newcommand{\kappag}{\kappa_{\nabla f}}
\newcommand{\kappagr}{\kappa_{\partial r}}
\newcommand{\kappac}{\kappa_{c}}
\newcommand{\kappav}{\kappa_{v}}
\newcommand{\kappaJ}{\kappa_J}
\newcommand{\kappaPhi}{\kappa_{\Phi}}
\newcommand{\sigmabar}{\bar\sigma}
\newcommand{\gammabar}{\bar\gamma}
\newcommand{\epsilonbar}{\bar\epsilon}
\newcommand{\taubar}{\bar\tau}
\newcommand{\taubarmin}{\taubar_{\min}}
\newcommand{\xstar}{x_*}
\newcommand{\ystar}{y_*}

\newcommand{\Itrue}{I_k^{\mathrm{true}}}
\newcommand{\Isucc}{I_k^{\text{succ}}}
\newcommand{\Ia}{I_k^{\Acal}}
\newcommand{\Ibara}{\bar{I}_k^{\Acal}}
\newcommand{\Proj}{\operatorname{Proj}}
\newcommand{\xq}[1]{{\color{magenta} #1}}

\newcommand{\uktrue}{u_k^{\mathrm{true}}}
\newcommand{\Uktrue}{U_k^{\mathrm{true}}}
\newcommand{\sktrue}{s_k^{\mathrm{true}}}
\newcommand{\Sktrue}{S_k^{\mathrm{true}}}
\newcommand{\yktrue}{y_k^{\mathrm{true}}}
\newcommand{\Yktrue}{Y_k^{\mathrm{true}}}
\newcommand{\zktrue}{z_k^{\mathrm{true}}}
\newcommand{\Zktrue}{Z_k^{\mathrm{true}}}
\newcommand{\grktrue}{g_{r,k}^{\text{true}}}
\newcommand{\Grktrue}{G_{r,k}^{\text{true}}}

\newcommand{\alphabar}{\bar{\alpha}}

\section{Introduction}\label{sec.introduction}

In this paper, we consider the constrained continuous optimization problem
\begin{equation}\label{prob:main}
  \min_{x\in\R{n}} \ f(x) + r(x) \ \  \text{subject to} \ \ c(x) = 0,\ \ x \in\Omega := \{x\in\R{n}:x\geq 0\},
\end{equation}
where $f : \R{n} \to \R{}$ defined by $f(x):=\E_{\Xi\sim \Pcal}[f(x,\Xi)]$ for all $x \in \R{n}$ is continuously differentiable, the regularization function $r: \R{n} \to \R{}$ is convex, and the equality constraint function $c : \R{n} \to \R{m}$ has $m\leq n$ and is continuously differentiable.  Here, the random variable~$\Xi$ is defined by a probability space $(\Ecal, \Fcal, \Pcal)$, where $\Ecal$ is a sample space, $\Fcal$ is a $\sigma$-algebra of measurable subsets of the sample space, and $\Pcal: \Fcal \to [0,1]$ is a probability measure. For each $x\in \R{n}$, the function $f(x,\cdot): \Ecal \to \R{}$ is measurable with respect to the $\sigma$-algebra $\Fcal$ on $\Ecal$. We presume that the expectation in the definition of $f$ is approximated using samples drawn from $\Xi$; i.e., we assume the objective $f$ and its derivatives are not known and must be estimated through sampling.  Problems of the form~\eqref{prob:main} arise naturally in a wide range of applications, including stochastic programming~\cite{bertsimas2018robust,shapiro2021lectures}, model compression~\cite{chen2018constraint}, and constrained deep learning~\cite{tomar2014manifold,zhu2019physics}, among others. For a comprehensive list of applications, we refer readers to~\cite{idrees2025constrained,jin2022stochastic,shi2025momentum} and the references therein.

Recently, several algorithms have been proposed for solving problem~\eqref{prob:main}, including augmented Lagrangian methods~\cite{jin2022stochastic,shi2025momentum}, a proximal-point method~\cite{boob2023stochastic}, a proximal-gradient method~\cite{boob2025level}, and a successive convex approximation method~\cite{idrees2025constrained}. Their theoretical guarantees have been established in terms of expected iteration complexity, i.e., the number of iterations required for a stationarity measure to fall below a prescribed accuracy (in expectation). Results in expectation, however, do not control tail events: the iterations required to reach an approximate stationary point can still be arbitrarily large with nonzero probability. To address this limitation, a separate line of work~\cite{berahas2025sequential,cutkosky2021high,gorbunov2024high,jain2019making,li2020highpro,sadiev2023high} has developed high-probability iteration complexity bounds that characterize the number of iterations needed to reach an approximate stationary point with a prescribed failure probability. These guarantees have been established in a variety of settings, including nonconvex stochastic optimization~\cite{cutkosky2021high,jain2019making,li2020highpro}, nonsmooth stochastic optimization~\cite{gorbunov2024high}, equality constrained stochastic optimization~\cite{berahas2025sequential}, and stochastic variational inequality problems~\cite{sadiev2023high}. To the best of our knowledge, however, no previously existing method for problem~\eqref{prob:main} provides high-probability iteration complexity guarantees.



Another limitation of existing high-probability analyses~\cite{gorbunov2024high,jain2019making,liu2024high} is the reliance on a stochastic gradient unbiasedness assumption. This excludes without-replacement schemes such as random reshuffling~\cite{gurbuzbalaban2021random}, which are standard in practice but yield biased mini-batch gradients within an epoch~\cite{mishchenko2020random,nguyen2018sgd}. Together, these limitations motivate the development of an algorithm for solving~\eqref{prob:main} that achieves high-probability iteration complexity guarantees without having to assume that the stochastic gradient estimator is unbiased.

In this paper, we design and analyze a proximal-gradient method for solving problem~\eqref{prob:main}. Our method can be seen as an extension of the proximal-gradient methods in~\cite{curtis2025proximal,dai2025proximal} to the setting where the objective function and its gradient are inaccessible and must be estimated, or as an extension of the step-search method~\cite{jin2024high} to stochastic optimization problems with both nonsmooth regularization and general constraints. The proposed method employs a step decomposition strategy and uses a sufficient decrease condition for a chosen merit function to determine step acceptance or rejection. The algorithm features two key aspects.
\begin{itemize}
    \item We propose a new stochastic first-order oracle (see Assumption~\ref{ass:stochastic_first_order_oracle}) in which the required gradient-estimation accuracy is governed by both a stationarity measure and the constraint violation.
    \item To address the challenges arising from noise in objective function evaluations, we employ a relaxed sufficient decrease condition similar to~\cite{jin2024high}. Under reasonable assumptions, we prove that, with high probability, the proposed algorithm generates an  $\epsilon$-KKT point within $\Ocal(\epsilon^{-2})$ iterations. Furthermore, the high-probability result implies both an almost-sure termination guarantee (Corollary~\ref{cor:convergence_expected_complexity}(i)) and an  expected iteration complexity bound (Corollary~\ref{cor:convergence_expected_complexity}(ii)).
\end{itemize}

Before proceeding, it is prudent that we clarify the contributions of this paper relative to two other papers in the literature.
\begin{itemize}
\item The problem studied in~\cite{curtis2025proximal} is similar to our problem~\eqref{prob:main} but they assume that the function $f$ and its gradient are known exactly. To address inexactness, we use a modified line search and an entirely different analysis (see next bullet).  
\item The problem studied in~\cite{jin2024high} is an \emph{unconstrained} optimization problem where the objective function is assumed to be differentiable, and possibly nonconvex.  Our problem~\eqref{prob:main} is distinct since it allows for the use of a nonsmooth regularization function and the presence of constraints. The overall framework for analyzing their method is the same framework that we adopt. However, since the two papers consider very different optimization problems, and the optimization algorithms analyzed are completely different, we overcome various challenges in our analysis that were not an issue in \cite{jin2024high}, such as needing to define an appropriate first-order oracle that takes constraint infeasibility into account and defining and analyzing an appropriate stopping time that considers both the constraints and nonsmooth regularizer.
\end{itemize}

\subsection{Organization}
In Section~\ref{sec.algorithm}, we provide the description of the proximal-gradient algorithm. In Section~\ref{sec.analysis}, we derive a high-probability iteration complexity result, which implies both an almost-sure termination result and expected iteration complexity result. Finally, we give concluding remarks in Section~\ref{sec.conclusion}.

\subsection{Notation}
Let $\R{}$ denote the set of real numbers, $\R{n}$ the set of $n$-dimensional real vectors, and $\N{}$ the set of natural numbers. For any $a \in \R{}$, let $\R{}_{>a}$ ($\R{}_{\ge a}$) denote the set of real numbers strictly larger than (larger than or equal to) $a$. We use $k \in \N{}$ as the iteration counter of the algorithm. Given a matrix $A$, we let $\sigmamin(A)$ denote the smallest singular value of $A$. Finally, we let $\text{exp}(a) := e^a$. 

\subsection{Standing assumption}

The following assumption is made throughout the paper. 

\bassumption\label{ass:general}
Let $\mathcal{X} \subseteq \R{n}$ be an open convex set containing every realization of the iterate, iterate plus normal step, and iterate plus full step (i.e., trial iterate) sequences of the algorithm. 
The objective function $f : \R{n} \rightarrow \R{}$ is continuously differentiable and bounded below over $\mathcal{X}$. The objective gradient $\nabla f : \R{n} \rightarrow \R{n}$ is Lipschitz continuous over $\mathcal{X}$. The regularization function $r : \R{n} \to \R{}$ is convex, Lipschitz continuous, and bounded below over $\mathcal{X}$. The constraint function $c : \R{n} \rightarrow \R{m}$ is continuously differentiable and bounded over $\mathcal{X}$, and constraint gradient $\nabla c_i : \R{n} \rightarrow \R{n}$ for all $i \in \{1,\ldots,m\}$ is Lipschitz continuous and bounded over $\mathcal{X}$.
\eassumption

Under Assumption~\ref{ass:general}, there exist constants $(f_{\inf}, r_{\inf}, \kappa_{\partial r},\kappa_{c},\kappa_{J}, L_g, L_J) \in \R{} \times \R{} \times \R{}_{> 0} \times \R{}_{>0} \times \R{}_{> 0} \times \R{}_{> 0} \times \R{}_{> 0}$ such that for all $x \in \Xcal$, one has
\begin{equation}\label{eq:consequence-of-basic-assumption}
  f(x) \geq f_{\inf},\ r(x) \geq r_{\inf},\ d_H(\partial r(x),0) \leq \kappa_{\partial r},\ \|c(x)\|_2 \leq \kappa_c,\ \text{and}\ \|\nabla c(x)\|_2 \leq \kappa_J
\end{equation}
where $d_H(\partial r(x),0)$ denotes the $\ell_2$-norm Hausdorff distance between the set $\partial r(x)$ and the origin set $\{0\}$, and for all $(x, \xbar) \in \Xcal \times \Xcal$, one has
\begin{align}\label{def:Lip-constants}
    \|\nabla f(x) - \nabla f(\xbar)\|_2 \leq L_g \|x-\xbar\|_2 \ \ \text{and} \ \ \|\nabla c(x)^T - \nabla c(\xbar)^T\|_2 \leq L_J \|x-\xbar\|_2.
\end{align}
All parameters used in this paper (including those  above) are gathered in Table~\ref{tab:parameters}.

\begin{table}[htbp]
  \centering
  \caption{Parameters and where they are defined.}
  \label{tab:parameters}
  \begin{tabular}{cc|cc|cc}
    \toprule
    Param. & Def. & Param. & Def. & Param. & Def. \\
    \midrule
     $\Xcal$ & Ass.~\ref{ass:general}  &
     $f_{\inf}$, $r_{\inf}$, $\kappa_{\partial r}$, $\kappa_{c}$, $\kappa_{J}$  & \eqref{eq:consequence-of-basic-assumption} &
     $L_g$, $L_J$  & \eqref{def:Lip-constants}    \\
     $\gamma$, $\epsilonbar_f$  & Alg.~\ref{alg:main} &
     $\kappa_G$  & Ass.~\ref{ass:kappaG}   &
     $\omega_1$, $\omega_2$, $\kappa_T$, $\alphabar$ & Def.~\ref{def:stoppingtime} \\
     $\kappav$ & \eqref{subprob:trust-region} &
     $\sigmamin$ & Ass.~\ref{ass:strong-licq} &
     $\kappa_1$ & Lem.~\ref{lem:lin-infeas-reduction} \\
     $\eta$, $\ibar$, $\theta$  & Def.~\ref{def:stoppingtime} &
     $\nu$, $b$, $\epsilon_f$ & Ass.~\ref{ass:zeroth_order_oracle} &
     $\epsilon_g$, $\kappa$, $\delta$ & Ass.~\ref{ass:stochastic_first_order_oracle} \\
     $u_f$, $p$ & Ass.~\ref{ass:p} &
    $\eta_1$ & \eqref{def:tauktrial}  &
    $\Wbar_0$, $\that$ & Thm.~\ref{thm:main} \\
    $\epsilon_\tau$ & \eqref{eq:tau-update} &
    $\tau_{\min}$ & Lem.~\ref{lem:merit-bounded} &
     &  \\
    \bottomrule
  \end{tabular}
\end{table}

\section{Algorithm}\label{sec.algorithm}

Our proposed method is stated as Algorithm~\ref{alg:main} on page~\pageref{alg:main}. For ease of development, we state the algorithm in the context of a realized iterate sequence $\{x_k\}$ and the computations that are performed for each $k$. In the following section, we analyze the algorithm with respect to the stochastic process that it defines.

During the $k$th iteration, a trial step $s_k = v_k + u_k$ is computed, where $v_k$ reduces linearized infeasibility of the equality constraints while maintaining feasibility with respect to $\Omega$, and $u_k$ aims to reduce the objective function while preserving the gains achieved by $v_k$. We then decide whether to accept the new trial iterate $x_k + s_k$ by considering the decrease in a noisy estimate of a merit function.


Given the $k$th iterate $x_k$ and proximal parameter $\alpha_k$, the normal step $v_k$ is computed as an approximate solution to the trust-region subproblem
\begin{equation}\label{subprob:trust-region}
  \min_{v \in \R{n}} m_k(v)\ \ \text{subject to}\ \ \|v\|_2 \leq \kappa_v \alpha_k \delta_k, \ \ x_k + v \in \Omega,
\end{equation}
where $m_k(v) := \tfrac{1}{2}\|c(x_k) + \nabla c(x_k)^T v\|_2^2$, $\kappa_v \in \R{}_{>0}$ is a user-defined constant, and   
\begin{equation}\label{eq:delta}
  \delta_k
  := \|\mathrm{Proj}_{T_\Omega(x_k)}(-\nabla c(x_k) c(x_k))\|_2.
\end{equation}
Here, $T_\Omega(x_k)$ is the tangent cone to $\Omega$ at $x_k$ (see \cite[Definition~2.4]{curtis2025proximal} for details) and $\text{Proj}_{T_\Omega(x_k)}(\cdot)$ is the orthogonal projection operator onto the set $T_\Omega(x_k)$. An approximate solution $v_k$ to subproblem~\eqref{subprob:trust-region} is required to satisfy
\begin{equation}\label{subprob:trust-region_conditions}
  \|v_k\|_2 \leq \kappa_v \alpha_k \delta_k,
  \quad x_k + v_k \in \Omega,
  \quad \text{and} \quad
  m_k(v_k) \leq m_k(v_k^c),
\end{equation}
where $v_k^c$ denotes the Cauchy point of subproblem~\eqref{subprob:trust-region} (see \cite[Section~3]{curtis2025proximal}). Throughout the paper, the computation of $v_k$ is assumed to satisfy the following.
\begin{assumption}~\label{ass:vk}
Given an iterate $x_k$ and proximal parameter $\alpha_k$, a deterministic procedure is used to compute a normal step $v_k$ satisfying~\eqref{subprob:trust-region_conditions}. 
\end{assumption}
Given the $k$th iterate $x_k$, proximal parameter $\alpha_k$, normal step $v_k$, and gradient estimate $g_k \approx \nabla f(x_k)$, the step $u_k$ is computed as the unique solution to
\begin{equation}\label{subprob:tangential}
  \min_{u \in \R{n}}\ \ q_k(v_k + u)\ \ \text{subject to}\
\ \nabla c(x_k)^T u = 0,\ \ x_k + v_k + u \in\Omega,
\end{equation}
with $q_k(v_k + u) := g_k^T(v_k + u) + \tfrac{1}{2\alpha_k} \|v_k + u\|_2^2 + r(x_k + v_k + u)$, where the constraint $\nabla c(x_k)^T u = 0$ ensures that $u_k$ lies in the null space of $\nabla c(x_k)^T$, so any gain in linearized feasibility achieved by $v_k$ is preserved.  From the optimality conditions for~\eqref{subprob:tangential}, there exist $g_{r,k} \in \partial r(x_k + v_k + u_k)$, $y_k \in \R{m}$, and $z_k \in \R{n}$ such that
\begin{subequations}\label{kkt:u}
\begin{align}
g_k + \tfrac{1}{\alpha_k}u_k + \tfrac{1}{\alpha_k}v_k + g_{r,k} + \nabla c(x_k) y_k + z_k &= 0, \label{kkt:stationary} \\
\nabla c(x_k)^T u_k &= 0, \ \ \text{and} \label{kkt:nullspace} \\
\|\min\{x_k + v_k + u_k,\, -z_k\}\|_2 &= 0. \label{kkt:complementarity}
\end{align}
\end{subequations}
We write $s_k = v_k + u_k$ for the trial step. The next assumption complements Assumption~\ref{ass:vk} and concerns the computation of $g_{r,k}$, $y_k$, and $z_k$.

\begin{assumption}\label{ass:grk-deterministic}
  Given an iterate $x_k$, proximal parameter $\alpha_k$, and normal step $v_k$ satisfying Assumption~\ref{ass:vk}, a deterministic procedure is used to compute the vectors $u_k$, $g_{r,k}\in\partial r(x_k+v_k+u_k)$, $y_k$, and $z_k$ satisfying~\eqref{kkt:u}.
\end{assumption}

Upon computation of the aforementioned quantities, an update is considered in the value of a merit parameter. Toward this end, the algorithm first sets, for a user-defined parameter $\eta_1 \in (0,1/3)$, the quantity
\begin{equation}\label{def:tauktrial}
  \tau_{k,\text{trial}} \gets \left \{\begin{array}{ll} \infty & \text{if } q_k(s_k) - q_k(0) \leq 0, \\ \tfrac{(1-2\eta_1)(\|c(x_k)\|_2 - \|c(x_k) + \nabla c(x_k)^T s_k\|_2)}{q_k(s_k) - q_k(0)} & \text{otherwise,} 
\end{array} \right.
\end{equation}
and then computes, for some $\epsilon_\tau\in(0,1)$, the value of the $k$th merit parameter as
\begin{equation}\label{eq:tau-update}
\tau_k \gets
\begin{cases}
   \tau_{k-1} & \text{if } \tau_{k-1} \leq \tau_{k,\text{trial}}, \\
   \min\{(1-\epsilon_{\tau})\tau_{k-1},  \tau_{k,\text{trial}}\} & \text{otherwise.} 
\end{cases}
\end{equation}
Subsequently, the algorithm estimates whether the trial iterate $x_k+s_k$ decreases the merit function
\begin{equation}\label{eq:merit_function-repeat}
  \phi(x,\tau) = \tau\big( f(x) + r(x)\big) + \|c(x)\|_2,
\end{equation}
which balances the objective function with violation of the equality constraints. Since exact evaluation of $f$ is assumed to be impossible, the algorithm computes independent estimates $f_k \approx f(x_k)$ and $f_k^+ \approx f(x_k+s_k)$, which in turn allows it to compute $\phi_k \approx \phi(x_k,\tau_k)$ and $\phi_k^+ \approx \phi(x_k+s_k,\tau_k)$.  The trial step acceptance test in Line~\ref{line:sufficient-decreasing} compares the change in the stochastic merit function values against a certain predicted model decrease, with a tolerance of $2\epsilonbar_f$ added to account for the noise in the estimates of $f$. If the test passes, the trial step is accepted and the value of the proximal parameter is increased by a factor $\gamma^{-1}$ subject to imposing that it never grows larger than $\alpha_0$; otherwise, the iterate remains unchanged and the value of the proximal parameter is reduced by a factor $\gamma$.

\begin{algorithm}[h]
  \caption{Stochastic Proximal-Gradient Method}\label{alg:main}
  \begin{algorithmic}[1]
    \State \textbf{Input:} $x_0 \in \Omega$, $(\alpha_0,\tau_{-1}) \in (0,1]^2$, $\epsilonbar_f \in \R{}_{>0}$, $\eta_1\in (0, 1/3)$, and $(\gamma,\epsilon_\tau) \in (0,1)^2$.
    \For{$k = 0, 1, 2, \ldots$}
      \State Compute $v_k$ satisfying~\eqref{subprob:trust-region_conditions} as an approximate solution to subproblem~\eqref{subprob:trust-region}.
      \State Generate an estimate $g_k\approx \nabla f(x_k)$.
      \State Compute $u_k$ as the unique solution to subproblem~\eqref{subprob:tangential}.
      \State Set $s_k \gets v_k + u_k$.
      \State Set $\tau_k$ using~\eqref{eq:tau-update}.\label{line:merit-parameter}
      \State Set $x_k^+ \gets x_k + s_k$.
      \State Compute estimates $f_k\approx f(x_k)$ and $f_k^+ \approx f(x_k^+)$.
      \State Set $\phi_k \gets \tau_k\bigl[f_k + r(x_k)\bigr] + \|c(x_k)\|_2$ and $\phi_k^+ \gets \tau_k\bigl[f_k^+ + r(x_k^+)\bigr] + \|c(x_k^+)\|_2$. \label{line:phi}
      \If{$\phi_k^+ - \phi_k
            \leq -\eta_1\!\left[
              \tau_k \frac{\|s_k\|_2^2}{2\alpha_k}
              + (\|c(x_k)\|_2 - \|c(x_k) + \nabla c(x_k)^T v_k\|_2)
            \right] + 2\epsilonbar_f$}~\label{line:sufficient-decreasing}
        \State Set $x_{k+1} \gets x_k^+$ and $\alpha_{k+1} \gets \min\{\alpha_0, \gamma^{-1}\alpha_k\}$.
      \Else
        \State Set $x_{k+1} \gets x_k$ and $\alpha_{k+1} \gets \gamma \alpha_k$.
      \EndIf
    \EndFor
  \end{algorithmic}
\end{algorithm}

\section{Analysis}\label{sec.analysis}

We now commence our analysis of Algorithm~\ref{alg:main}. In Section~\ref{sec:stochastic-process}, we describe the stochastic process defined by Algorithm~\ref{alg:main} along with related assumptions and definitions, and then analyze a stopping time for the process in Section~\ref{sec:analyze-the-stochastic-process}.

\subsection{Stochastic Process and Related Assumptions and Definitions}\label{sec:stochastic-process}

Algorithm~\ref{alg:main} uses estimates $g_k \approx \nabla f(x_k)$, $f_k \approx f(x_k)$, and $f_k^+ \approx f(x_k+s_k)$ during each iteration $k$. Therefore, to analyze the algorithm we must first describe how these estimates are computed and the assumptions we make about these estimates, and then introduce notation that describes the overall stochastic process. 

We begin by defining the stochastic process primarily in terms of quantities that are actually computed by the algorithm, and  later introduce additional sequences of random variables that are defined by the process that are useful for analyzing it. For each $k\in\N{}$, let $M_k =(\Xi_{k,g},\Xi_{k,f},\Xi_{k,f}^+)$ be a triple of independent random variables each with the same distribution as $\Xi$ used to define  $f(\cdot)$.  From this triple, we presume that the algorithm makes use of samples $(\xi_{k,g},\xi_{k,f},\xi_{k,f}^+)$ and defines
\begin{equation}\label{def:gk}
    g_k := \nabla f(x_k,\xi_{k,g}), \ \ f_k := f(x_k,\xi_{k,f}), \ \ \text{and} \ \ f_k^+ := f(x_k+s_k,\xi_{k,f}^+). 
\end{equation}
We define the associated estimation errors as
\begin{equation}\label{def:ek}
e_k := |f_k - f(x_k)| \ \ \text{and} \ \  e_k^+ := |f_k^+ - f(x_k^+)|.
\end{equation}
For the analysis only, it will be convenient to define the following quantities.

\begin{definition}[True primal-dual solution]\label{def:true-quantities}
  For each $k\in\N{}$, given $x_k$ and $v_k$, the true primal-dual solution to~\eqref{subprob:tangential}, i.e., the one obtained when $g_k$ is replaced by $\nabla f(x_k)$, is $(\uktrue,\grktrue,\yktrue,\zktrue)$, and in addition $\sktrue := v_k + \uktrue$.
\end{definition}

Associated with $\{M_k\}_{k\geq 0}$,  we have  filtrations $\{\Fcal_k\}_{k \ge 0}$ and $\{\Fcal_{k+1/3}\}_{k \ge 0 }$ where 
\begin{align*}
\Fcal_0 &= \sigma(x_0,\alpha_0), &
\Fcal_k &= \sigma(x_0,\alpha_0,M_0, M_1, \ldots, M_{k-1}) \ \ \text{for $k\geq 1$,}  \\ 
\Fcal_{1/3} &= \sigma(x_0,\alpha_0,\Xi_{0,g}), &  
\Fcal_{k+1/3} &= \sigma(x_0,\alpha_0,M_0, M_1, \ldots, M_{k-1},\Xi_{k,g}) \ \ \text{for all $k\geq 1$,}
\end{align*}
where $x_0$ is the initial iterate (presumed fixed), $\alpha_0$ is the initial proximal parameter (also fixed), and $\sigma(\cdot)$ is the sigma algebra generated by its arguments.
%
%
%
%
Overall, we analyze the stochastic process defined, for all $k \in\N{}$, by the random variables 
\begin{align*}
  \{X_k, \Delta_k, V_k^c, V_k, A_k\}, &\ \text{which are $\Fcal_k$-measurable,} \\
  \{G_k, U_k, \Uktrue, S_k, \Sktrue, Y_k, &\Yktrue, Z_k, \Zktrue, G_{r,k}, \Grktrue, \Tcal_{k,\rm trial}, \Tcal_k, X_k^+\}, \\
  &\ \text{which are $\Fcal_{k+1/3}$-measurable, and}  \\
  \{\Phi_k, \Phi_k^+, F_k, F_k^+,E_k, E_k^+\}, &\ \text{which are $\Fcal_{k+1}$-measurable,}
\end{align*}
where for each $k \in \N{}$ these quantities have the aforementioned realizations
\begin{align*}
&\{x_k, \delta_k, v_k^c, v_k, \alpha_k\}, \\
&\{g_k, u_k, \uktrue, s_k,\sktrue,y_k, \yktrue, z_k, \zktrue, g_{r,k},  \grktrue, \tau_{k,\text{trial}}, \tau_k, x_k^+\}, \ \text{and}\\
&\{\phi_k, \phi_k^+,f_k,f_k^+,e_k,e_k^+\}.
\end{align*} 


We make the following pair of assumptions. For all $k \in \N{}$, we write $\P_\bullet[\cdot | \Fcal_k]$ to denote probability of event $\cdot$ with respect to the distribution of the random variable $\bullet$ conditioned on $\Fcal_k$, and similarly define $\E_\bullet[\cdot | \Fcal_k]$ as the expected value of $\cdot$ with respect to the distribution of the random variable $\bullet$ conditioned on $\Fcal_k$.

\begin{assumption}[Stochastic zeroth-order oracle]\label{ass:zeroth_order_oracle}
There exists $(\nu,b,\epsilon_f) \in \R{3}_{>0}$ such that, for all $k \in \N{}$, the estimation error at $X_k$ satisfies
\begin{equation}\label{eq:subexp-E}
\begin{aligned}
    \E_{\Xi_{k,f}}[E_k | \Fcal_k]  &\leq \epsilon_f
    \quad \text{and} \\ 
    \E_{\Xi_{k,f}}\!\left[
      \exp\!\left\{
        \lambda\bigl(E_k - \E_{\Xi_{k,f}}[E_k|\Fcal_k]\bigr)
      \right\} | \Fcal_k
    \right]
    &\leq \exp\!\(\frac{\lambda^2 \nu^2}{2}\) \ \ \text{for all} \ \ \lambda \in \left[0, \frac{1}{b}\right],
\end{aligned}
\end{equation}
and the estimation error at $X_k^+$ satisfies 
\begin{equation}\label{eq:subexp-E+}
\begin{aligned}
    \E_{(\Xi_{k,g},\Xi_{k,f}^+)}[E_k^+ | \Fcal_k] 
    \leq&\ \epsilon_f
    \quad \text{and} \\ 
    \E_{(\Xi_{k,g},\Xi_{k,f}^+)}\!\left[
      \exp\!\left\{
        \lambda\bigl(E_k^+ - \E_{(\Xi_{k,g},\Xi_{k,f}^+)}[E_k^+|\Fcal_k]\bigr)
      \right\} \Big| \Fcal_k
    \right]
    \leq&\ \exp\!\(\frac{\lambda^2 \nu^2}{2}\) \\ &\ \text{for all} \ \ \lambda \in \left[0, \frac{1}{b}\right].
  \end{aligned}
  \end{equation}  
\end{assumption}


\begin{assumption}[Stochastic first-order oracle]\label{ass:stochastic_first_order_oracle}
  There exists $(\epsilon_g, \kappa) \in \R{2}_{>0}$ and $\delta\in[0,1/2)$ such that,
  for all $k\in\N{}$, the following inequality holds:
    \begin{align}
    \P_{\Xi_{k,g}} \Big[
      & \|G_k - \nabla f(X_k)\|_2 \label{eq:first_order_accuracy} \\
      & \leq \max\left\{
        \epsilon_g,
        \kappa\sqrt{\|S_k\|_2^2 + \|c(X_k)\|_2 - \|c(X_k) + \nabla c(X_k)^T V_k\|_2}
      \right\}
    \Big| \Fcal_k
    \Big]
    \geq 1 - \delta. \nonumber
    \end{align}
\end{assumption}

\begin{remark} The following comments concern Assumption~\ref{ass:zeroth_order_oracle} and Assumption~\ref{ass:stochastic_first_order_oracle}.
  \begin{enumerate}
  \item
  The latter conditions in~\eqref{eq:subexp-E} and~\eqref{eq:subexp-E+} are one-sided subexponential-like assumptions on $E_k$ and $E_k^+$, respectively. For more context, let us first recall the definition of sub-Gaussian and subexponential random variables. A random variable $X$ is called sub-Gaussian if its moment-generating function satisfies $\E[\exp(\lambda (X-\E[X]))] \leq \exp(\nu^2 \lambda^2/2)$ for all $\lambda \in \R{}$ for some $\nu > 0$. Such a random variable has light tails that decay at least as fast as those of a Gaussian. Subexponential random variables relax this by requiring the moment-generating function bound to hold only for $|\lambda| \leq 1/b$ for some $b > 0$, thereby accommodating heavier tails. A one-sided subexponential-like random variable further relaxes this requirement by imposing the bound only for $\lambda \in [0, 1/b]$.
  \item The conditions in~\eqref{eq:subexp-E+} are in some sense stronger than those in~\eqref{eq:subexp-E} since the bounds are assumed to hold for the estimation error $E_k^+$, which is realized after the computation of a trial step $S_k$. Thus, they relate to estimation errors at the trial iterate $X_k^+$, the value of which depends on $\Fcal_k$-measurable $X_k$, the normal step $V_k$, and the tangential step $U_k$ to produce $X_k^+ = X_k + S_k = X_k + V_k + U_k$.
  \item Similarly, with respect to \eqref{eq:first_order_accuracy}, a realization of $S_k$ is only known after a realization of $G_k$ is known.  Therefore, the accuracy condition~\eqref{eq:first_order_accuracy} is a posterior requirement that couples the gradient estimation error with the trial step.
  \item The two quantities inside the max of~\eqref{eq:first_order_accuracy} serve distinct roles. The constant~$\epsilon_g$ bounds irreducible noise that is assumed to be intrinsic to the oracle. The second term in the max couples the required accuracy to the norm of the trial step $\|S_k\|_2$, which reflects progress towards stationarity, and to the reduction in linearized infeasibility $\|c(X_k)\|_2 - \|c(X_k) + \nabla c(X_k)^T V_k\|_2$, which reflects progress towards feasibility. Together, they permit larger gradient estimation errors when significant progress is possible in either stationarity or feasibility, and demand higher accuracy when the potential progress is small.
  \end{enumerate}
\end{remark}

We also make the following assumption, which depends on the parameters defined in Assumptions~\ref{ass:zeroth_order_oracle} and Assumption~\ref{ass:stochastic_first_order_oracle}, where the quantities $u_f$ and $p$ will be used in the stopping-time analysis in the next subsection.

\begin{assumption}\label{ass:p}
With respect to the algorithmic parameter $\epsilonbar_f \in \R{}_{>0}$, the real number
$$
u_f := \epsilonbar_f - \epsilon_f
$$
satisfies, 
with $(\nu,b)$ from Assumption~\ref{ass:zeroth_order_oracle} and $\delta$ from Assumption~\ref{ass:stochastic_first_order_oracle},
\begin{equation}\label{def:p}
u_f \geq 0 \ \ \text{and} \ \
p := 1 - \delta - \exp\(-\min\left\{\frac{u_f^2}{2\nu^2}, \frac{u_f}{2b}\right\}\) \in (1/2,1).
\end{equation}
\end{assumption}
Note that $u_f \geq 0$ whenever $\epsilonbar_f\geq \epsilon_f$, i.e., when the noise level $\epsilonbar_f$ used in the line search of Algorithm~\ref{alg:main} is at least as large as the actual noise bound $\epsilon_f$ in \eqref{eq:subexp-E}. At the same time, the condition $p \in (1/2,1]$ in~\eqref{def:p} will hold if $u_f$ is sufficiently large relative to $\nu$ and $b$ and if $\delta$ is sufficiently close to zero.

Our analysis of Algorithm~\ref{alg:main} uses a classification of the iterations based on the quality of the estimates $g_k$, $f_k$, and $f_k^+$ and whether the Armijo-type inequality (see Line~\ref{line:sufficient-decreasing} in Algorithm~\ref{alg:main}) holds. We call an iteration \emph{true} if the estimates are sufficiently accurate, and we call an iteration \emph{successful} if the trial step satisfies the Armijo-type inequality. We formalize these definitions next.

\begin{definition}[True iteration]\label{def:true}
Iteration $k \in \N{}$ is \emph{true} if and only if
  \begin{align*}
    \|G_k - \nabla f(X_k)\|_2 &\leq \max\left\{\epsilon_g, \kappa \sqrt{\|S_k\|_2^2 + \|c(X_k)\|_2 - \|c(X_k) + \nabla c(X_k)^TV_k\|_2} \right\} \\
    \text{and}\ \ E_k + E^+_k &\leq 2\epsilonbar_f.
  \end{align*}
\end{definition}

\begin{definition}[Successful iteration]\label{def:successful} Iteration $k \in \N{}$ is \emph{successful} if and only if the inequality in Line~\ref{line:sufficient-decreasing} of Algorithm~\ref{alg:main} holds.
\end{definition}

We can now define the $\Fcal_{k+1}$-measurable random variables
\begin{equation*}
I_k := 
\begin{cases}
    1 & \text{if iteration $k$ is true}, \\
    0 & \text{otherwise,}
\end{cases}
\ \ \text{and} \ \
\Theta_k :=
\begin{cases}
    1 & \text{if iteration $k$ is successful}, \\
    0 & \text{otherwise,}
\end{cases}
\end{equation*}
and realizations of these random variables will be denoted as $i_k$ and $\theta_k$, respectively.

The above classifications are used to prove a relaxed decrease condition on the merit function for iterations that are \emph{true} and \emph{successful} (see Lemma~\ref{lem:alg-properties}). The proof uses the $\Fcal_{k+1/3}$-measurable random variable
\bequationNN
  W_k:= \Tcal_k (f(X_k) - f_{\text{inf}} + r(X_k) - r_{\text{inf}}) + \|c(X_k)\|_2,
\eequationNN
where a realization of $W_k$ will be denoted by $w_k$.  The random variable $W_k$ is a shifted version of the merit function in~\eqref{eq:merit_function-repeat}, which combines objective function gaps (to their independent infima) with the violation of the equality constraint.

We now introduce a constraint qualification assumption. The one that we adopt uses the active set at the Cauchy point for the normal step subproblem defined earlier.  Specifically, it uses the $\Fcal_k$-measurable random set
$$
\Acal_k^c := \{i \in [n]: [X_k+V_k^c]_i = 0 \},
$$
which is the index set of variables active at $X_k+V_k^c$ with respect to $x\in\Omega$.

\begin{assumption}\label{ass:strong-licq}
There exists $\sigmamin \in \R{}_{>0}$ such that $\sigmamin([\nabla c(X_k), I_{\Acal_k^c}^T]^T) \geq \sigmamin$ for all $k \in \N{}$, where $I_{\Acal_k^c}$ denotes the subset of rows of the $n\times n$ identity matrix that correspond to the elements in $\Acal_k^c$.
\end{assumption}


Finally, we assume that the stochastic gradient estimates are bounded uniformly.  This assumption holds, for example, if $f$ is defined as a sample average approximation.  It is critical for proving that $T_k$ is bounded away from zero.

\begin{assumption}\label{ass:kappaG}
There exists $\kappa_G \in \R{}_{>0}$ such that $\|G_k\|_2 \leq \kappa_G$ for all $k \in \N{}$.
\end{assumption}

\subsection{Stopping Time Analysis}\label{sec:analyze-the-stochastic-process}

Our aim in this subsection is to analyze a stopping time of our proposed algorithm. The stopping time that we consider is defined in terms of quantities introduced in the previous subsection along with a uniform lower bound for the merit parameter. Due to this latter dependence, before defining our stopping time, we prove a few introductory lemmas to show that such a lower bound exists.

We start by providing a lower bound on the reduction in linearized infeasibility. 

\begin{lemma}\label{lem:lin-infeas-reduction}
  There exists $\kappa_1 \in \R{}_{>0}$ such that, for all $k \in \N{}$, one has
  \begin{align*}
    &\ \|c(X_k)\|_2 - \|c(X_k) + \nabla c(X_k)^T S_k\|_2 \\
    =&\ \|c(X_k)\|_2 - \|c(X_k) + \nabla c(X_k)^T V_k\|_2 \\
    \geq&\ \kappa_1 \sigmamin^2 \|c(X_k)\|_2 \min \left\{ \frac{1}{1 + \|\nabla c(X_k) \nabla c(X_k)^T\|_2}, \kappav A_k \right\} \\
    \geq&\ \frac{\kappa_1}{\kappa_c} \sigmamin^2 \|c(X_k)\|_2^2 \min\left\{\frac{1}{1+\|\nabla c(X_k) \nabla c(X_k)^T\|_2}, \kappav A_k \right\}.
  \end{align*}
\end{lemma}
\begin{proof}
  The proof follows exactly as in~\cite[Lemma~5.10]{curtis2025proximal} and uses Assumption~\ref{ass:strong-licq}.  \qed
\end{proof}

The next lemma bounds $\Delta_k$ in terms of the constraint violation.

\begin{lemma}\label{lem:vk(1)-bound}
  For all $k \in \N{}$, one has 
$\Delta_k 
\leq 2\kappaJ \|c(X_k)\|_2 \leq 2\kappa_J \kappa_c$.
\end{lemma}
\begin{proof}
  See~\cite[Lemma 5.3]{curtis2025proximal}, \eqref{eq:consequence-of-basic-assumption}, and the definition of $\delta_k$ (so $\Delta_k$) in~\eqref{eq:delta}. \qed
\end{proof}

We next prove that the merit parameters are uniformly bounded from zero.

\begin{lemma}\label{lem:merit-bounded}
  For all $k\in\N{}$, it holds that 
  $$
  \Tcal_{k,\text{trial}} \geq \Tcal_k \geq 
  \tau_{\min} := \min\left\{\tau_{-1},(1-\epsilon_\tau) \frac{(1-2\eta_1)\kappa_1\sigmamin^2 \min \left\{\frac{1}{1+ \kappaJ^2} , \kappa_v\right\}}{2(\kappa_G + \kappa_{\partial r}) \kappav \kappaJ + 2\kappav^2 \kappaJ^2 \kappac}\right\}
  > 0.
  $$
\end{lemma}
\begin{proof}
  Let us begin by bounding the change in $q_k$ with respect to the step $S_k$.  
  Since~$r$ is convex, it follows from the subgradient inequality at $X_k + V_k$ that
  \begin{equation}\label{eq:subgradient.inequality}
    r(X_k + V_k) - r(X_k) \leq (G_{r,k}^v)^T V_k\ \ \text{for all} \ \ G_{r,k}^v \in \partial r(X_k + V_k).
  \end{equation}
  Since $S_k = V_k+U_k$, $U_k$ is optimal for \eqref{subprob:tangential}, and the zero vector is feasible for~\eqref{subprob:tangential},
  \begin{align*}
    &\ G_k^T S_k + \tfrac{1}{2A_k}\|S_k\|_2^2 + r(X_k+S_k) \\
    =&\ G_k^T(V_k+U_k) + \tfrac{1}{2A_k}\|V_k+U_k\|_2^2 + r(X_k+V_k+U_k) \\
    =&\ q_k(V_k+U_k) \leq q_k(V_k) = G_k^T V_k + \tfrac{1}{2A_k}\|V_k\|_2^2 + r(X_k+V_k).
  \end{align*}
  Subtracting $q_k(0) = r(X_k)$ from both sides and using~\eqref{eq:subgradient.inequality} gives
  \begin{equation}\label{bound-denom-from-zero}
    \begin{aligned}
      & G_k^T S_k + \tfrac{1}{2A_k}\|S_k\|_2^2 + r(X_k+S_k) - r(X_k) \\
      \leq \ & G_k^T V_k + \tfrac{1}{2A_k}\|V_k\|_2^2 + r(X_k+V_k) - r(X_k) \\
      \leq \ & G_k^T V_k + \tfrac{1}{2A_k}\|V_k\|_2^2 +(G_{r,k}^v)^T V_k \\
      \overset{(i)}{\leq} \ & (\|G_k\|_2 + \|G_{r,k}^v\|_2) \|V_k\|_2 + \tfrac{1}{2A_k} \|V_k\|_2^2 \\
      \overset{(ii)}{\leq} \ & (\|G_k\|_2 + \|G_{r,k}^v\|_2) \kappa_v A_k \Delta_k + \tfrac{1}{2A_k} \kappa_v^2 A_k^2 \Delta_k^2 \\
      \overset{(iii)}{=} \ & (\|G_k\|_2 + \|G_{r,k}^v\|_2) \kappa_v A_k \Delta_k + \thalf\kappa_v^2 A_k \Delta_k^2 \\
      \overset{(iv)}{\leq} \ & (\|G_k\|_2 + \|G_{r,k}^v\|_2) 2 \kappav A_k \kappaJ \|c(X_k)\|_2 + 2 \kappav^2 A_k \kappaJ^2 \kappac \|c(X_k)\|_2 \\
      \overset{(v)}{\leq} \ & 2(\kappa_G + \kappa_{\partial r}) \kappav \kappaJ A_k \|c(X_k)\|_2 + 2\kappav^2 \kappaJ^2 \kappac A_k \|c(X_k)\|_2,
    \end{aligned}
  \end{equation}
  where (i) follows from the Cauchy-Schwarz inequality, (ii) follows from the first inequality in~\eqref{subprob:trust-region_conditions}, (iii) follows from canceling an $A_k$ from the second term, (iv)~follows from Lemma~\ref{lem:vk(1)-bound} and~\eqref{eq:consequence-of-basic-assumption}, and (v)~follows from~\eqref{eq:consequence-of-basic-assumption} and Assumption~\ref{ass:kappaG}.

  We now prove that $\{\Tcal_{k,\rm trial}\}_{k\geq 0}$ is uniformly bounded away from zero. Conditioned on $q_k(S_k) - q_k(0) \leq 0$, one finds $\Tcal_{k,\rm trial} = \infty$.  Otherwise, conditioned on $q_k(S_k) - q_k(0) > 0$, we first show that $c(X_k) \neq 0$. To see this, note that if $c(X_k) = 0$ then it would follow from~\eqref{eq:delta} that $\Delta_k = 0$, and then $V_k = 0$ follows from~\eqref{subprob:trust-region_conditions}, and then $q_k(S_k) = q_k(V_k+U_k) \leq q_k(0)$ follows from how $U_k$ is computed in~\eqref{subprob:tangential}, which is a contradiction.  Thus, conditioned on $q_k(S_k) - q_k(0) > 0$, it must hold that $c(X_k) \neq 0$, which can be combined with the definition of $\Tcal_{k,\rm trial}$, $S_k = V_k + U_k$, $\nabla c(X_k)^T U_k = 0$ (see~\eqref{kkt:nullspace}), Lemma~\ref{lem:lin-infeas-reduction}, \eqref{bound-denom-from-zero},  \eqref{eq:consequence-of-basic-assumption}, and the fact that $A_k \leq 1$ to get
  \begin{align} 
    \Tcal_{k,\text{trial}} 
    &= \frac{(1-2\eta_1)(\|c(X_k)\|_2 - \|c(X_k) + \nabla c(X_k)^T S_k\|_2)}{q_k(S_k) - q_k(0)} \nonumber \\
    &\geq \frac{(1-2\eta_1) \kappa_1 \sigmamin^2 \|c(X_k)\|_2 \min \left\{\frac{1}{1 + \|\nabla c(X_k) \nabla c(X_k)^T\|_2}, \kappa_v A_k\right\}}{2(\kappa_G + \kappa_{\partial r}) \kappav \kappaJ A_k \|c(X_k)\|_2 + 2 \kappav^2 \kappaJ^2 \kappac A_k \|c(X_k)\|_2} \nonumber \\
    &\geq \frac{(1-2\eta_1)\kappa_1\sigmamin^2 \min \left\{\frac{1}{1+ \kappaJ^2}, \kappa_v A_k\right\}}{2(\kappa_G + \kappa_{\partial r}) \kappav \kappaJ A_k + 2\kappav^2 \kappaJ^2 \kappac A_k} \nonumber \\
    &\geq \frac{(1-2\eta_1)\kappa_1\sigmamin^2 \min \left\{\frac{1}{1+ \kappaJ^2} , \kappa_v\right\}}{2(\kappa_G + \kappa_{\partial r}) \kappav \kappaJ + 2\kappav^2 \kappaJ^2 \kappac} > 0. \label{eq:tauktrial-bound}
  \end{align}
  Combining the merit parameter update rule~\eqref{eq:tau-update} and~\eqref{eq:tauktrial-bound} completes the proof. \qed
\end{proof}



We now define a stopping time.
\begin{definition}[Stopping time]\label{def:stoppingtime}
Choose any $\eta \in \(0, \ 
  \frac{\sqrt{2\eta_1(1-3\eta_1)}}
       {1 + 2\sqrt{2\eta_1(1-3\eta_1)}}
\) \subset (0,\tfrac{1}{2})$, and then define (see the discussion below showing $\omega_2 > 0$) the quantities 
\begin{align*}
    \omega_1 &:= \frac{1}{1-\eta_1}\(\frac{\tau_{\min} (L_g + 2\eta_1) + L_J}{\tau_{\min}} + \frac{\kappa^2}{2\eta_1}\) > 0, \\
    \omega_2 &:= \frac{\tau_{\min} L_g + L_J}{\(1 - \dfrac{\eta^2}{2\eta_1(1-2\eta)^2} - 3\eta_1\)\tau_{\min}} > 0, \\ 
    \kappa_T &:= \sqrt{\kappa_c}/\sqrt{\kappa_1\sigma_{\min}^2 \min\{1/(1+\kappaJ^2), \kappav\}} > 0, \ \ \text{and} \\
    \alphabar &:= \gamma^{\ibar+1}\alpha_0 \ \ \text{with $\ibar$ defined as the smallest nonnegative integer} \\
    &\qquad\qquad\quad\ \text{such that $\gamma^{\ibar}\alpha_0 \leq \min\{\omega_1^{-1},\omega_2^{-1}\}$},
\end{align*}
where we note that this ensures that $\alphabar \leq \gamma \alpha_0 < \alpha_0$.
Then, for $\theta\in(0,1)$ and any
\begin{equation}\label{eq:epsilon_lower_bound}
\begin{aligned}
\epsilon 
&> \max\{\sqrt{2},\kappa_T\}
\max\left\{
\tfrac{\epsilon_g}{\eta}, 
\max\left\{
1+\kappa,
\tfrac{1}{\sqrt{1 - 2\eta}} \right\} \sqrt{\tfrac{(4 + 4\tau_{-1})\epsilonbar_f}{\theta\tau_{\min}\alphabar\eta_1(p - 1/2)}} \right\}
\end{aligned}
\end{equation}
let the stopping time be the random variable
\bequation\label{eq.T}
  T_\epsilon := \min \left\{ k \in \N{} : \Psi_k \leq \frac{\epsilon}{\max\{\sqrt{2},\kappa_T\}} \right\},
\eequation
where, for each $k \in \N{}$, the optimality measure is defined as
\bequationNN
  \Psi_k := \sqrt{\frac{\|\Sktrue\|_2^2}{2A_k^2} + \frac{\|\min\{X_k, -\Zktrue\}\|_2^2}{2} + \frac{\|c(X_k)\|_2 - \|c(X_k) + \nabla c(X_k)^T V_k\|_2}{A_k}}.
\eequationNN
We allow $T_\epsilon = \infty$ whenever the minimum in \eqref{eq.T} is not attained.
\end{definition}

In Definition~\ref{def:stoppingtime}, it is straightforward to see that $\omega_1 > 0$ and $\kappa_T > 0$, and once it is shown that $\omega_2 > 0$ it is also straightforward to see that $\alphabar > 0$.  To show $\omega_2 > 0$, first let $\abar := \sqrt{2\eta_1(1-3\eta_1)} > 0$ so that we may rewrite the interval containing $\eta$ as $\eta\in(0,\tfrac{\abar}{1+2\abar})$.  To prove that $\omega_2 > 0$, it is sufficient to show that
$$
\frac{\eta^2}{2\eta_1(1-2\eta)^2}< 1 - 3\eta_1 = \frac{\abar^2}{2\eta_1}.
$$
Multiplying both sides by $2\eta_1(1-2\eta)^2 > 0$ yields
$$
\eta^2 < \abar^2(1-2\eta)^2
\ \iff\
\eta < \abar(1-2\eta)
\ \iff \ 
(1+2\abar)\eta < \abar
\ \iff \ 
\eta < \frac{\abar}{1+2\abar},
$$
which holds by the definition of $\eta$, thus confirming that $\omega_2 > 0$.

We now justify the definition of the stopping time $T_\epsilon$.  To do this, the next lemma shows that, for a given run of Algorithm~\ref{alg:main}, an iterate satisfying the stopping time condition is an $\epsilon$-KKT point, which we define first.

\begin{definition}\label{def:epskkt}
The $k$th iterate of Algorithm~\ref{alg:main} is an $\epsilon$-KKT point if and only if
\begin{equation*}
\max\left\{\|\nabla f(x_k) + g_{r,k}^{\text{true}}+ \nabla c(x_k)y_k^{\text{true}} + z_k^{\text{true}}\|_2, \|\min\{x_k, -z_k^{\text{true}}\}\|_2, \|c(x_k)\|_2\right\} \leq \epsilon
\end{equation*}
where we also remind the reader that $g_{r,k}^{\text{true}} \in \partial r(x_k + s_k^{\text{true}})$.
\end{definition}

\begin{lemma}\label{lem:kkt-result} 
If during iteration $\kbar \in \N{}$ the tuple $(x_{\kbar}, \alpha_{\kbar}, v_{\kbar}, u_{\kbar}^{\text{true}}, g_{r,\kbar}^{\text{true}}, y_{\kbar}^{\text{true}}, z_{\kbar}^{\text{true}}, s_{\kbar}^{\text{true}})$ satisfies the inequality defining $T_\epsilon$, then 
$x_{\kbar}$ is an $\epsilon$-KKT point.
\end{lemma}
\begin{proof}
  Consider such a $\kbar$. It follows from Definition~\ref{def:true-quantities}, the fact that the stopping time condition ensures $\|s_{\kbar}^{\text{true}}/(\sqrt{2}\alpha_{\kbar})\|_2 \leq \epsilon / \max\{\sqrt{2},\kappa_T\}$ (and similarly for the other terms in the left-hand side of the stopping-time condition),  Lemma~\ref{lem:lin-infeas-reduction}, \eqref{eq:consequence-of-basic-assumption}, $A_k \leq 1$, and Definition~\ref{def:stoppingtime} that
  \begin{equation*}
    \begin{aligned}
      \|\nabla f(x_{\kbar}) + g_{r,\kbar}^{\text{true}}+ \nabla c(x_{\kbar})y_{\kbar}^{\text{true}} + z_{\kbar}^{\text{true}}\|_2 = \left\|s_{\kbar}^{\text{true}}/\alpha_{\kbar}\right\|_2
      &\leq \tfrac{\sqrt{2}\epsilon}{\max\{\sqrt{2},\kappa_T\}} \leq \epsilon, \\
      \|\min\{x_{\kbar}, -z_{\kbar}^{\text{true}}\}\|_2 &\leq \tfrac{\sqrt{2}\epsilon}{\max\{\sqrt{2},\kappa_T\}} \leq \epsilon, \ \ \text{and} \\
      \|c(x_{\kbar})\|_2 &\leq \tfrac{\kappa_T\epsilon}{\max\{\sqrt{2},\kappa_T\}} \leq \epsilon. 
    \end{aligned}
  \end{equation*}
  It follows from these inequalities and Definition~\ref{def:epskkt} that $x_{\kbar}$ is an $\epsilon$-KKT point. \qed
\end{proof}

Before commencing the main part of our analysis, we motivate the definition of~$\alphabar$ in Definition~\ref{def:stoppingtime} by proving the following technical lemma and subsequently stating a definition for ``large'' proximal-parameters.

\begin{lemma}\label{lem:A-wd}
For all $k\in\N{}$,
$\min\{A_k,A_{k+1}\} < \alphabar$ if and only if $\max\{A_k,A_{k+1}\} \leq \alphabar$.
\end{lemma}
\begin{proof}
  Observe that Definition~\ref{def:stoppingtime} ensures that $\bar\alpha \leq \gamma \alpha_0$. This fact, along with the updating strategy for the proximal parameter, ensure that if either $\min\{A_k,A_{k+1}\} < \alphabar$ or $\max\{A_k,A_{k+1}\} \leq \alphabar$, then
  \begin{equation}\label{min-less-than-max}
    \max\{A_k,A_{k+1}\} = \gamma^{-1}\min\{A_k,A_{k+1}\} > \min\{A_k,A_{k+1}\}.
  \end{equation}
  Thus, if $\max\{A_k,A_{k+1}\} \leq \alphabar$, then it follows from~\eqref{min-less-than-max} that $\min\{A_k,A_{k+1}\} < \max\{A_k,A_{k+1}\} \leq \alphabar$. Otherwise, if $\min\{A_k,A_{k+1}\} < \alphabar$, then it follows from how $\alphabar$ is defined that $\gamma^{-1}\min\{A_k,A_{k+1}\} \leq \alphabar$. Combining this inequality with the equation in~\eqref{min-less-than-max} shows that $\max\{A_k,A_{k+1}\} \leq \alphabar$, which completes the proof.
  \qed
\end{proof}

The lemma above can now be used to define the following random variable.

\begin{definition}[Large proximal parameters]\label{def:largestep}
For each $k \in\N{}$, let
\[
L_k :=
\begin{cases}
1, & \text{if } \min\{A_k, A_{k+1}\} \geq \bar{\alpha}, \\
0, & \text{if } \max\{A_k, A_{k+1}\} \leq \bar{\alpha},
\end{cases}
\]
which is well defined because of Lemma~\ref{lem:A-wd}.  We also note that $L_k$ is $\Fcal_{k+1}$-measurable since $A_k$ is $\Fcal_k$-measurable and $\Theta_k$ is $\Fcal_{k+1}$-measurable.
\end{definition}

We now start the main part of our stopping-time analysis. The next result bounds the difference between $S_k$ and $S_k^{\mathrm{true}}$, and between $U_k$ and $U_k^{\mathrm{true}}$.
\begin{lemma}\label{lem:uk-gk-relation}
For $k\in\N{}$, $\tfrac{1}{A_k} \|S_k - S_k^{\mathrm{true}}\|_2 = \tfrac{1}{A_k} \|U_k - U_k^{\mathrm{true}}\|_2 \leq \|G_k - \nabla f(X_k)\|_2$.
\end{lemma}
\begin{proof}
  The claimed equality follows from $\|S_k - S_k^{\mathrm{true}}\|_2 = \|V_k + U_k - (V_k + U_k^{\mathrm{true}})\|_2 = \|U_k - U_k^{\mathrm{true}}\|_2$.  The claimed inequality follows trivially if $\|U_k - U_k^{\mathrm{true}}\|_2 = 0$, so for the remainder of the proof we assume that $\|U_k - U_k^{\mathrm{true}}\|_2 \neq 0$. Consider arbitrary $k \in \N{}$ and define $h_k : \R{n} \to \R{}$ by
  \begin{align*}
    h_k(u) = r(X_k + V_k + u) + I_\Omega(X_k + V_k + u),
  \end{align*}
  where $I_\Omega$ denotes the indicator function associated with the set $\Omega$. Note that $h_k$ is convex since both $r$ and $I_\Omega$ are convex.  One can now write $U_k$ and $U_k^{\mathrm{true}}$, respectively, as the unique solutions to the following subproblems:
  \begin{align*}
    U_k^{\mathrm{true}} = \arg\min_{u \in \R{n}} &\ \nabla f(X_k)^T u + \tfrac{1}{2A_k}\|u\|_2^2 + \tfrac{1}{A_k}V_k^T u + h_k(u) \\ \text{subject to} &\ \nabla c(X_k)^T u = 0 \ \ \text{and} \\
    U_k = \arg\min_{u \in \R{n}} &\ G_k^T u + \tfrac{1}{2A_k}\|u\|_2^2 + \tfrac{1}{A_k} V_k^T u + h_k(u) \\ \text{subject to} &\ \nabla c(X_k)^T u = 0.
  \end{align*}
  Using the subdifferential sum rule \cite{rockafellar1970convex}, which holds since the constraint qualification $\mathrm{dom}(r) \cap \mathrm{int}\bigl(\mathrm{dom}(I_\Omega)\bigr) \neq \emptyset$ holds trivially since $\mathrm{dom}(r) = \R{n}$, it follows that for all $u \in \R{n}$ one has
  \begin{align*}
    \partial h_k(u) = \partial r(X_k+V_k+u) + N_\Omega(X_k+V_k+u).
  \end{align*} 
  Here, $N_\Omega(X_k+V_k+u)$ is the normal cone to $\Omega$ at $X_k+V_k+u$. 
  From the KKT conditions of the optimization problems above, there exist $(G_{h,k}^{\text{true}}, G_{h,k}) \in \partial h_k(U_k^{\mathrm{true}}) \times \partial h_k(U_k)$ and $(Y_k^{\mathrm{true}}, Y_k) \in \R{m} \times \R{m}$ such that
  \begin{align*}
    \nabla f(X_k) + \tfrac{1}{A_k}(U_k^{\mathrm{true}} + V_k) + G_{h,k}^{\text{true}} + \nabla c(X_k) Y_k^{\mathrm{true}} &= 0 \ \ \text{and} \\  
    G_k + \tfrac{1}{A_k}(U_k + V_k) + G_{h,k} + \nabla c(X_k) Y_k &= 0. 
  \end{align*}
  Subtracting these two equations gives
  \begin{equation*}
    (G_k - \nabla f(X_k)) + \tfrac{1}{A_k}(U_k - U_k^{\mathrm{true}}) + (G_{h,k} - G_{h,k}^{\text{true}}) + \nabla c(X_k)(Y_k - Y_k^{\mathrm{true}}) = 0.
  \end{equation*}
  Multiplying by $(U_k - U_k^{\mathrm{true}})^T$ and using $\nabla c(X_k)^T U_k^{\mathrm{true}} = \nabla c(X_k)^T U_k = 0$ gives
  \begin{equation*}
    (G_k - \nabla f(X_k))^T (U_k - U_k^{\mathrm{true}}) + \tfrac{1}{A_k}\|U_k - U_k^{\mathrm{true}}\|_2^2 + (G_{h,k} - G_{h,k}^{\text{true}})^T(U_k - U_k^{\mathrm{true}}) = 0.
  \end{equation*}
  Since $h_k$ is convex, its subdifferential $\partial h_k$ is a maximal monotone operator, from which it follows that $(G_{h,k} - G_{h,k}^{\text{true}})^T(U_k - U_k^{\mathrm{true}}) \geq 0$. Combining this with the equation above and using the Cauchy-Schwarz inequality gives
  \begin{align*}
    \tfrac{1}{A_k}\|U_k - U_k^{\mathrm{true}}\|_2^2 &\leq -(G_k - \nabla f(X_k))^T (U_k - U_k^{\mathrm{true}}) \\
    &\leq \|G_k-\nabla f(X_k)\|_2 \|U_k - U_k^{\mathrm{true}}\|_2.
  \end{align*}
  Dividing both sides by $\|U_k - U_k^{\mathrm{true}}\|_2 \neq 0$ completes the proof. \qed
\end{proof}

We now prove an upper bound on the change in the merit function attained by the trial step $S_k$ in all iterations. 
\begin{lemma}\label{lem:merit-reduction}
  For all $k \in \N{}$, it holds that
  \begin{align*}
    &\ \phi(X_k + S_k,\Tcal_k) - \phi(X_k,\Tcal_k) \\
    \leq&\ -\tfrac{\Tcal_k}{2A_k}\|S_k\|_2^2 + \tfrac{\Tcal_k L_g + L_J}{2} \|S_k\|_2^2 + \Tcal_k(\nabla f(X_k) - G_k)^TS_k + \Tcal_kG_k^TS_k \\
    &\ + \Tcal_k(r(X_k+S_k) - r(X_k)) + \tfrac{\Tcal_k}{2A_k} \|S_k\|_2^2 + \|c(X_k) + \nabla c(X_k)^TV_k\|_2 - \|c(X_k)\|_2.
  \end{align*}
\end{lemma}
\begin{proof}
  Consider arbitrary $k \in \N{}$. Lipschitz continuity of $\nabla f(\cdot)$ and $\nabla c(\cdot)^T$ (see Assumption~\ref{ass:general}), $S_k = V_k + U_k$, and $\nabla c(X_k)^TU_k = 0$ show that 
  \begin{align*}
    &\ \phi(X_k + S_k,\Tcal_k) - \phi(X_k,\Tcal_k) \\
    =&\ \Tcal_k(f(X_k + S_k) + r(X_k + S_k))+ \|c(X_k + S_k)\|_2 \\
    &\ - (\Tcal_k(f(X_k)+r(X_k))+ \|c(X_k)\|_2) \\
    \leq&\ \Tcal_k(\nabla f(X_k)^TS_k + \tfrac{L_g}{2}\|S_k\|_2^2) + \Tcal_k(r(X_k+S_k) - r(X_k)) \\
    &\ + \|c(X_k)+\nabla c(X_k)^TS_k\|_2 - \|c(X_k)\|_2 + \tfrac{L_J}{2} \|S_k\|_2^2 \\
    =&\ -\tfrac{\Tcal_k}{2A_k}\|S_k\|_2^2 + \tfrac{\Tcal_k L_g + L_J}{2} \|S_k\|_2^2 + \Tcal_k(\nabla f(X_k) - G_k)^TS_k + \Tcal_kG_k^TS_k \\
    &\ + \Tcal_k(r(X_k+S_k) - r(X_k)) + \tfrac{\Tcal_k}{2A_k} \|S_k\|_2^2 + \|c(X_k)+\nabla c(X_k)^TV_k\|_2 - \|c(X_k)\|_2,
  \end{align*} 
  which completes the proof.
  \qed
\end{proof}

Lemma~\ref{lem:merit-reduction} provides an upper bound on the change in the merit 
function. The next lemma bounds the term $\Tcal_k(\nabla f(X_k) - G_k)^TS_k$ appearing in this upper bound in terms of measures related to stationarity.

\begin{lemma}\label{lem:gradient_error_bound}
  Conditioned on $k < T_\epsilon$, the following hold.
  \begin{itemize}
    \item[(i)] If $\|G_k - \nabla f(X_k)\|_2 \leq \kappa \sqrt{ \|S_k\|_2^2 + \|c(X_k)\|_2 - \|c(X_k) + \nabla c(X_k)^T V_k\|_2}$, then
  \end{itemize}
  \begin{align*}
    \Tcal_k(\nabla f(X_k) - G_k)^TS_k \leq&\  \(\tfrac{\kappa^2}{4\eta_1} + \eta_1\) \Tcal_k A_k \tfrac{\|S_k\|_2^2}{A_k} \\
    &\ + \eta_1 (\|c(X_k)\|_2 - \|c(X_k) + \nabla c(X_k)^T V_k\|_2).
  \end{align*}
  \begin{itemize}
    \item[(ii)] If $\|G_k - \nabla f(X_k)\|_2 \leq \epsilon_g$, then
  \end{itemize}
  \begin{align*}
    \Tcal_k(\nabla f(X_k) - G_k)^TS_k \leq&\  \(\tfrac{\eta^2}{4\eta_1(1-2\eta)^2} + \eta_1\)\Tcal_k\tfrac{\|S_k\|_2^2}{A_k} \\
    &\ + \eta_1(\|c(X_k)\|_2 - \|c(X_k) + \nabla c(X_k)^T V_k\|_2).
  \end{align*}
\end{lemma}
\begin{proof}
  Condition on the event that $k < T_\epsilon$. First, suppose that the condition in~(i) holds. Then, along with the Cauchy-Schwarz inequality and the fact that $\nabla c(X_k)^T U_k = 0$ (see~\eqref{kkt:nullspace}) shows that
  \begin{align}
      &\ \Tcal_k(\nabla f(X_k) - G_k)^T S_k \nonumber \\ 
      \leq&\ \Tcal_k\|\nabla f(X_k) - G_k\|_2 \|S_k\|_2 \nonumber \\
      \leq&\ \kappa\Tcal_k \sqrt{\|S_k\|_2^2 + \|c(X_k)\|_2 - \|c(X_k) + \nabla c(X_k)^T V_k\|_2} \|S_k\|_2. \label{eq:part-i-cs}
  \end{align}
  Now $\Tcal_k \leq 1$ and Young's inequality with $\delta_Y \in \R{}_{>0}$ (also known as the Peter-Paul inequality), which is that $ab \leq \frac{a^2}{2\delta_Y} + \frac{\delta_Y b^2}{2}$ for any $(a,b) \in \R{}_{\geq0} \times \R{}_{\geq0}$, shows with $\delta_Y = \tfrac{2\eta_1}{\kappa}$, $a = \|S_k\|_2$, and $b = \sqrt{\|S_k\|_2^2 + \|c(X_k)\|_2 - \|c(X_k) + \nabla c(X_k)^T V_k\|_2}$ that
  \begin{equation*}
    \begin{aligned}
      & \ \kappa\Tcal_k \|S_k\|_2 \sqrt{\|S_k\|_2^2 + \|c(X_k)\|_2 - \|c(X_k) + \nabla c(X_k)^T V_k\|_2} \\
      \leq&\ \kappa \Tcal_k\( \tfrac{\|S_k\|_2^2}{2\delta_Y} + \tfrac{\delta_Y}{2}\bigl(\|S_k\|_2^2 + \|c(X_k)\|_2 - \|c(X_k) + \nabla c(X_k)^T V_k\|_2\bigr) \) \\
      =&\ \kappa \Tcal_k \(\tfrac{1}{2\delta_Y} + \tfrac{\delta_Y}{2}\) \|S_k\|_2^2 + \tfrac{\kappa\Tcal_k\,\delta_Y}{2}(\|c(X_k)\|_2 - \|c(X_k) + \nabla c(X_k)^T V_k\|_2) \\
      =&\ \(\tfrac{\kappa^2}{4\eta_1} + \eta_1\) \Tcal_k \|S_k\|_2^2 + \eta_1 \Tcal_k (\|c(X_k)\|_2 - \|c(X_k) + \nabla c(X_k)^T V_k\|_2) \\[4pt]
      \leq&\ \(\tfrac{\kappa^2}{4\eta_1} + \eta_1 \) \Tcal_k A_k \tfrac{\|S_k\|_2^2}{A_k} + \eta_1 \bigl (\|c(X_k)\|_2 - \|c(X_k) + \nabla c(X_k)^T V_k\|_2\bigr).
    \end{aligned}
  \end{equation*}
  Combining this inequality with~\eqref{eq:part-i-cs} completes the proof of part~(i).

  On the other hand, suppose the condition in (ii) holds. From $A_k \leq 1$, $\|S_k^{\mathrm{true}}\|_2 \geq \|\min\{X_k,-Z_k^{\mathrm{true}}\}\|_2$ (see~\cite[Lemma~4.6]{curtis2025proximal}), and~\eqref{eq:epsilon_lower_bound}, one obtains
  \begin{equation}\label{eq:epsilon-inequality-v2}
    \sqrt{\frac{\|S_k^{\mathrm{true}}\|_2^2}{A_k^2} + \frac{\|c(X_k)\|_2 - \|c(X_k) + \nabla c(X_k)^T V_k\|_2}{A_k}} > \frac{\epsilon}{\max\{\sqrt{2}, \kappa_T\}} > \frac{\epsilon_g}{\eta},
  \end{equation}
  which implies along with the condition in (ii) that
  \begin{equation}\label{eq:grad-diff-first-v2}
    \|G_k - \nabla f(X_k)\|_2 \leq \epsilon_g < \eta\sqrt{\frac{\|S_k^{\mathrm{true}}\|_2^2}{A_k^2} + \frac{\|c(X_k)\|_2 - \|c(X_k) + \nabla c(X_k)^T V_k\|_2}{A_k}}.
  \end{equation}
  Thus, with the Cauchy-Schwarz inequality, Lemma~\ref{lem:uk-gk-relation}, and $\nabla c(X_k)^T U_k = 0$,
  \begin{align*}
    &\ \( \frac{\|S_k^{\mathrm{true}}\|_2^2}{A_k} + \|c(X_k)\|_2 - \|c(X_k) + \nabla c(X_k)^TV_k\|_2 \) \\
    &\ - \( \frac{\|S_k\|_2^2}{A_k} + \|c(X_k)\|_2 - \|c(X_k) + \nabla c(X_k)^TV_k\|_2 \) \\
    = & \ \frac{\|S_k^{\mathrm{true}}\|_2^2 - \|S_k\|_2^2}{A_k} \\
    = & \ \frac{1}{A_k} \( 2(S_k^{\mathrm{true}})^T(S_k^{\mathrm{true}} - S_k) - \|S_k^{\mathrm{true}} - S_k\|_2^2 \) \\
    \leq & \ \frac{2\|S_k^{\mathrm{true}}\|_2 \|S_k^{\mathrm{true}} - S_k\|_2}{A_k} \\
    \leq&\ 2\|S_k^{\mathrm{true}}\|_2 \|G_k - \nabla f(X_k)\|_2 \\
    \leq&\ 2\eta \|S_k^{\mathrm{true}}\|_2 \sqrt{\frac{\|S_k^{\mathrm{true}}\|_2^2}{A_k^2} + \frac{\|c(X_k)\|_2 - \|c(X_k) + \nabla c(X_k)^TV_k\|_2}{A_k}} \\
    \leq&\ 2\eta \cdot A_k \sqrt{\frac{\|S_k^{\mathrm{true}}\|_2^2}{A_k^2} + \frac{\|c(X_k)\|_2 - \|c(X_k) + \nabla c(X_k)^TV_k\|_2}{A_k}} \\
    &\quad\ \cdot \sqrt{\frac{\|S_k^{\mathrm{true}}\|_2^2}{A_k^2} + \frac{\|c(X_k)\|_2 - \|c(X_k) + \nabla c(X_k)^TV_k\|_2}{A_k}} \\
    =&\ 2\eta A_k \( \frac{\|S_k^{\mathrm{true}}\|_2^2}{A_k^2} + \frac{\|c(X_k)\|_2 - \|c(X_k) + \nabla c(X_k)^TV_k\|_2}{A_k} \) \\
    =&\ 2\eta \( \frac{\|S_k^{\mathrm{true}}\|_2^2}{A_k} + \|c(X_k)\|_2 - \|c(X_k) + \nabla c(X_k)^TV_k\|_2 \),
  \end{align*}
  which, after rearrangement, gives
  \begin{align*}
    &\ \frac{\|S_k^{\mathrm{true}}\|_2^2}{A_k} + \|c(X_k)\|_2 - \|c(X_k) + \nabla c(X_k)^T V_k\|_2 \\
    \leq&\ \frac{1}{1-2\eta} \(\frac{\|S_k\|_2^2}{A_k} + \|c(X_k)\|_2 - \|c(X_k) + \nabla c(X_k)^T V_k\|_2\).
  \end{align*}
  Dividing both sides by $A_k$ and taking square roots yields
  \begin{align}
    &\ \sqrt{\frac{\|S_k^{\mathrm{true}}\|_2^2}{A_k^2} + \frac{\|c(X_k)\|_2 - \|c(X_k) + \nabla c(X_k)^T V_k\|_2}{A_k}} \nonumber \\
    \leq&\ \frac{1}{\sqrt{1-2\eta}} \sqrt{\frac{\|S_k\|_2^2}{A_k^2} + \frac{\|c(X_k)\|_2 - \|c(X_k) + \nabla c(X_k)^T V_k\|_2}{A_k}}. \label{eq:true-to-computed-sqrt}
  \end{align}
  Combining the Cauchy--Schwarz inequality, \eqref{eq:grad-diff-first-v2}, \eqref{eq:true-to-computed-sqrt}, and the fact that $1-2\eta\in(0,1)$ (since $\eta \in (0,\tfrac12)$ under Definition~\ref{def:stoppingtime}), one finds that
  \begin{align}
    &\ \Tcal_k(\nabla f(X_k) - G_k)^T S_k \nonumber \\
    \leq&\ \Tcal_k\|\nabla f(X_k) - G_k\|_2 \|S_k\|_2 \nonumber  \\
    \leq&\ \eta\Tcal_k \|S_k\|_2 \sqrt{\frac{\|S_k^{\mathrm{true}}\|_2^2}{A_k^2} + \frac{\|c(X_k)\|_2 - \|c(X_k) + \nabla c(X_k)^T V_k\|_2}{A_k}} \nonumber \\
    \leq&\ \frac{\eta \Tcal_k}{\sqrt{1-2\eta}} \|S_k\|_2 \sqrt{\frac{\|S_k\|_2^2}{A_k^2} + \frac{\|c(X_k)\|_2 - \|c(X_k) + \nabla c(X_k)^T V_k\|_2}{A_k}} \nonumber \\
    \leq&\ \frac{\eta\Tcal_k}{1-2\eta} \|S_k\|_2 \sqrt{\frac{\|S_k\|_2^2}{A_k^2} + \frac{\|c(X_k)\|_2 - \|c(X_k) + \nabla c(X_k)^T V_k\|_2}{A_k}}. \label{eq:part-ii-pre-young}
  \end{align}
  To bound the right-hand side of~\eqref{eq:part-ii-pre-young}, one can apply Young's inequality with $\delta_Y = \tfrac{2A_k\eta_1(1-2\eta)}{\eta} > 0$, $a = \|S_k\|_2$, and $b = \sqrt{\frac{\|S_k\|_2^2}{A_k^2} + \frac{\|c(X_k)\|_2 - \|c(X_k) + \nabla c(X_k)^T V_k\|_2}{A_k}}$ and use the fact that $\Tcal_k \leq 1$ to obtain
  \begin{align*}
    & \ \frac{\eta \Tcal_k}{1-2\eta}\|S_k\|_2 \sqrt{\frac{\|S_k\|_2^2}{A_k^2} + \frac{\|c(X_k)\|_2 - \|c(X_k) + \nabla c(X_k)^T V_k\|_2}{A_k}} \\
    \leq& \ \frac{\eta \Tcal_k}{1-2\eta}\(
    \frac{\|S_k\|_2^2}{2\delta_Y} + \frac{\delta_Y}{2}\( \frac{\|S_k\|_2^2}{A_k^2} + \frac{\|c(X_k)\|_2 - \|c(X_k) + \nabla c(X_k)^T V_k\|_2}{A_k} \) \) \\
    =& \ \frac{\eta \Tcal_k}{1-2\eta}\(\frac{1}{2\delta_Y} + \frac{\delta_Y}{2A_k^2}\) \|S_k\|_2^2 + \frac{\tfrac{\eta \Tcal_k}{1-2\eta} \delta_Y}{2A_k} \(\|c(X_k)\|_2 - \|c(X_k) + \nabla c(X_k)^T V_k\|_2\) \\
    =& \ \dfrac{\eta\Tcal_k}{1 - 2\eta} \(\frac{\frac{\eta}{1-2\eta}}{4A_k\eta_1} + \frac{A_k\eta_1}{\frac{\eta}{1-2\eta} A_k^2}\) \|S_k\|_2^2 + \eta_1 \Tcal_k(\|c(X_k)\|_2 - \|c(X_k) + \nabla c(X_k)^T V_k\|_2) \\
    \leq& \ \(\frac{\eta^2 \Tcal_k}{4\eta_1(1-2\eta)^2} + \eta_1\Tcal_k\) \frac{\|S_k\|_2^2}{A_k} + \eta_1 \bigl(\|c(X_k)\|_2 - \|c(X_k) + \nabla c(X_k)^T V_k\|_2\bigr) \\
    =& \ \(\frac{\eta^2}{4\eta_1(1-2\eta)^2} + \eta_1\) \Tcal_k\frac{\|S_k\|_2^2}{A_k} + \eta_1\bigl(\|c(X_k)\|_2 - \|c(X_k) + \nabla c(X_k)^T V_k\|_2\bigr).
  \end{align*}
  Combining this inequality with~\eqref{eq:part-ii-pre-young} completes the proof of part~(ii).
  \qed
\end{proof}

We now use Lemma~\ref{lem:merit-reduction} and 
Lemma~\ref{lem:gradient_error_bound} to obtain a refined upper bound on 
the merit function decrease in iterations before the stopping time condition holds.
\begin{lemma}~\label{lem:merit-reduction-new}
  Conditioned on $k < T_\epsilon$, the following hold.
  \begin{itemize}
    \item[(i)] If $\|G_k - \nabla f(X_k)\|_2 \leq \kappa \sqrt{ \|S_k\|_2^2 + \|c(X_k)\|_2 - \|c(X_k) + \nabla c(X_k)^TV_k\|_2}$, then 
    \begin{align*}
      &\ \phi(X_k + S_k,\Tcal_k) - \phi(X_k,\Tcal_k) \\
      \leq &\ -\frac{\Tcal_k}{2A_k}\|S_k\|_2^2 + \frac{\Tcal_k L_g + L_J}{2} \|S_k\|_2^2 + \(\frac{\kappa^2}{4\eta_1} + \eta_1\) A_k \Tcal_k\frac{\|S_k\|_2^2}{A_k} \\
      &\ + \eta_1 (\|c(X_k) + \nabla c(X_k)^TV_k\|_2 - \|c(X_k)\|_2).
    \end{align*}
    \item[(ii)] If $\|G_k - \nabla f(X_k)\|_2 \leq \epsilon_g$, then 
    \begin{align*}
      &\ \phi(X_k + S_k,\Tcal_k) - \phi(X_k,\Tcal_k) \\
      \leq&\ -\frac{\Tcal_k}{2A_k}\|S_k\|_2^2 + \frac{\Tcal_k L_g + L_J}{2} \|S_k\|_2^2 + \(\frac{\eta^2}{4\eta_1(1-2\eta)^2} + \eta_1\) \Tcal_k \frac{\|S_k\|_2^2}{A_k} \\
      &\ + \eta_1 (\|c(X_k) + \nabla c(X_k)^TV_k\|_2 - \|c(X_k)\|_2).
    \end{align*}
  \end{itemize}
\end{lemma}
\begin{proof}
%
  Condition on the event that $k < T_\epsilon$. Let us first establish a bound that is common to both part (i) and part (ii). In particular, let us show that
  \begin{align}
    &\ \Tcal_k(G_k^T S_k + \tfrac{1}{2A_k}\|S_k\|_2^2 + r(X_k+S_k) - r(X_k)) \nonumber \\
    &\ + \|c(X_k) + \nabla c(X_k)^T V_k\|_2 - \|c(X_k)\|_2 \nonumber \\
    \leq&\ -2\eta_1(\|c(X_k)\|_2 - \|c(X_k) + \nabla c(X_k)^T V_k\|_2). \label{eq:vk-terms-bound}
  \end{align}
  Note that the term $\|c(X_k)\|_2 - \|c(X_k) + \nabla c(X_k)^T V_k\|_2$ in the right-hand side of~\eqref{eq:vk-terms-bound} is nonnegative due to Lemma~\ref{lem:lin-infeas-reduction}. Let us now consider the two cases in~\eqref{def:tauktrial}.
 
  \medskip
  \noindent\textit{Case~1:} Suppose that $G_k^T S_k + \tfrac{1}{2A_k}\|S_k\|_2^2 + r(X_k+S_k) - r(X_k) \leq 0$. Since $\eta_1 \in (0,1/3)$, one has $0 < 2\eta_1 < 1$, and therefore
  \begin{align*}
    &\ \Tcal_k(G_k^T S_k + \tfrac{1}{2A_k}\|S_k\|_2^2 + r(X_k+S_k) - r(X_k)) \\
    &\ + \|c(X_k) + \nabla c(X_k)^T V_k\|_2 - \|c(X_k)\|_2 \\
    \leq&\ -(\|c(X_k)\|_2 - \|c(X_k) + \nabla c(X_k)^T V_k\|_2) \\
    \leq&\ -2\eta_1(\|c(X_k)\|_2 - \|c(X_k) + \nabla c(X_k)^T V_k\|_2),
  \end{align*}
  thus verifying that~\eqref{eq:vk-terms-bound} holds in this case, as claimed. 

  \medskip
  \noindent\textit{Case~2:} Suppose that $G_k^T S_k + \tfrac{1}{2A_k}\|S_k\|_2^2 + r(X_k+S_k) - r(X_k) > 0$. It follows from~\eqref{def:tauktrial} and Lemma~\ref{lem:merit-bounded} that
  \begin{equation*}
    \Tcal_{k,\mathrm{trial}} = \frac{(1-2\eta_1)(\|c(X_k)\|_2 - \|c(X_k) + \nabla c(X_k)^T S_k\|_2)}{G_k^T S_k + \tfrac{1}{2A_k}\|S_k\|_2^2 + r(X_k+S_k) - r(X_k)}
  \end{equation*}
  is finite and strictly positive with $\Tcal_{k,\mathrm{trial}} \geq \Tcal_k > 0$. Combining this fact with the value of $\Tcal_{k,\mathrm{trial}}$ above, $G_k^T S_k + \tfrac{1}{2A_k}\|S_k\|_2^2 + r(X_k+S_k) - r(X_k) > 0$, and the fact that $\|c(X_k) + \nabla c(X_k)^T S_k\|_2 = \|c(X_k) + \nabla c(X_k)^T V_k\|_2$ (see~\eqref{kkt:nullspace}) gives
  \begin{align*}
    &\ \Tcal_k\bigl(G_k^T S_k + \tfrac{1}{2A_k}\|S_k\|_2^2 + r(X_k+S_k) - r(X_k)\bigr) \\
    \leq&\ (1-2\eta_1)(\|c(X_k)\|_2 - \|c(X_k) + \nabla c(X_k)^T V_k\|_2).
  \end{align*}
  Therefore, one finds that
  \begin{align*}
    &\ \Tcal_k (G_k^T S_k + \tfrac{1}{2A_k}\|S_k\|_2^2 + r(X_k+S_k) - r(X_k)) \\
    &\ + \|c(X_k) + \nabla c(X_k)^T V_k\|_2 - \|c(X_k)\|_2 \\
    \leq&\ (1-2\eta_1)(\|c(X_k)\|_2 - \|c(X_k) + \nabla c(X_k)^T V_k\|_2) \\
    &\ - (\|c(X_k)\|_2 - \|c(X_k) + \nabla c(X_k)^T V_k\|_2) \\
    =&\ -2\eta_1(\|c(X_k)\|_2 - \|c(X_k) + \nabla c(X_k)^T V_k\|_2),
  \end{align*}
  again verifying that~\eqref{eq:vk-terms-bound} holds. Overall, ~\eqref{eq:vk-terms-bound} holds in all cases, as claimed.

  Now suppose the condition in (i) holds. Lemma~\ref{lem:merit-reduction} gives
  \begin{align*}
    &\ \phi(X_k + S_k, \Tcal_k) - \phi(X_k, \Tcal_k) \\
    \leq&\ -\tfrac{\Tcal_k}{2A_k} \|S_k\|_2^2 + \tfrac{\Tcal_k L_g + L_J}{2}\|S_k\|_2^2 + \Tcal_k(\nabla f(X_k) - G_k)^T S_k \\
    &\ + \Tcal_k (G_k^T S_k + \tfrac{1}{2A_k} \|S_k\|_2^2 + r(X_k+S_k) - r(X_k)) \\
    &\ + \|c(X_k) + \nabla c(X_k)^T V_k\|_2 - \|c(X_k)\|_2.
  \end{align*}
  Using Lemma~\ref{lem:gradient_error_bound}(i) and~\eqref{eq:vk-terms-bound} gives
  \begin{align*}
    &\ \phi(X_k + S_k, \Tcal_k) - \phi(X_k, \Tcal_k) \\
    \leq&\ -\tfrac{\Tcal_k}{2A_k} \|S_k\|_2^2 + \tfrac{\Tcal_k L_g + L_J}{2} \|S_k\|_2^2 + \(\tfrac{\kappa^2}{4\eta_1} + \eta_1\) \Tcal_k A_k \tfrac{\|S_k\|_2^2}{A_k} \\
    &\ + \eta_1(\|c(X_k)\|_2 - \|c(X_k) + \nabla c(X_k)^T V_k\|_2) \\
    &\ - 2\eta_1(\|c(X_k)\|_2 - \|c(X_k) + \nabla c(X_k)^T V_k\|_2) \\
    =& -\tfrac{\Tcal_k}{2A_k} \|S_k\|_2^2 + \tfrac{\Tcal_k L_g + L_J}{2} \|S_k\|_2^2 + \(\tfrac{\kappa^2}{4\eta_1} + \eta_1\) A_k \Tcal_k \tfrac{\|S_k\|_2^2}{A_k} \\
    &\ + \eta_1(\|c(X_k) + \nabla c(X_k)^T V_k\|_2 - \|c(X_k)\|_2),
  \end{align*}
  which is the bound claimed in part~(i).

  Now suppose the condition in (ii) holds. Lemma~\ref{lem:merit-reduction}, Lemma~\ref{lem:gradient_error_bound}(ii), and~\eqref{eq:vk-terms-bound} give
  \begin{align*}
    &\ \phi(X_k + S_k, \Tcal_k) - \phi(X_k, \Tcal_k) \\
    \leq&\ -\tfrac{\Tcal_k}{2A_k} \|S_k\|_2^2 + \tfrac{\Tcal_k L_g + L_J}{2} \|S_k\|_2^2 + \(\tfrac{\eta^2}{4\eta_1(1-2\eta)^2} + \eta_1\) \Tcal_k \tfrac{\|S_k\|_2^2}{A_k} \\
    &\ + \eta_1(\|c(X_k)\|_2 - \|c(X_k) + \nabla c(X_k)^T V_k\|_2) \\
    &\ - 2\eta_1(\|c(X_k)\|_2 - \|c(X_k) + \nabla c(X_k)^T V_k\|_2) \\
    =&\ -\tfrac{\Tcal_k}{2A_k} \|S_k\|_2^2 + \tfrac{\Tcal_k L_g + L_J}{2} \|S_k\|_2^2 + \(\tfrac{\eta^2}{4\eta_1(1-2\eta)^2} + \eta_1\) \Tcal_k \tfrac{\|S_k\|_2^2}{A_k} \\
    &\ + \eta_1(\|c(X_k) + \nabla c(X_k)^T V_k\|_2 - \|c(X_k)\|_2),
  \end{align*}
  which is the bound claimed in part (ii).
  \qed
\end{proof}

Our next result is needed because the stopping time is defined using $S_k^{\mathrm{true}}$, whereas the algorithm only has access to $S_k$. 

\begin{lemma}\label{lem:sk-difference}
  Conditioned on $k < T_\epsilon$, the following hold.
  \begin{itemize}
    \item[(i)] If $\|G_k - \nabla f(X_k)\|_2 \leq \kappa \sqrt{ \|S_k\|_2^2 + \|c(X_k)\|_2 - \|c(X_k) + \nabla c(X_k)^TV_k\|_2}$, then
    \begin{align*}
      &\ \(\Tcal_k \frac{\|S_k\|_2^2}{2A_k} + \|c(X_k)\|_2 - \|c(X_k) + \nabla c(X_k)^T V_k\|_2\) \\
      \geq&\ \frac{1}{(1+\kappa)^2} \(\Tcal_k \frac{\|S_k^{\mathrm{true}}\|_2^2}{2A_k} + \|c(X_k)\|_2 - \|c(X_k) + \nabla c(X_k)^T V_k\|_2\).
    \end{align*}
    \item[(ii)] If $\|G_k - \nabla f(X_k)\|_2 \leq \epsilon_g$, then
    \begin{align*}
      &\ \(\Tcal_k \frac{\|S_k\|_2^2}{2A_k} + \|c(X_k)\|_2 - \|c(X_k) + \nabla c(X_k)^T V_k\|_2\) \\
      \geq&\ (1 - 2\eta) \(\Tcal_k \frac{\|S_k^{\mathrm{true}}\|_2^2}{2A_k} + \|c(X_k)\|_2 - \|c(X_k) + \nabla c(X_k)^T V_k\|_2\).
    \end{align*}
  \end{itemize}
\end{lemma}
\begin{proof}
  Condition on the event that $k < T_\epsilon$.  Suppose that the condition in (i) holds.  It follows from the condition in (i), the Cauchy-Schwarz inequality, $A_k \leq 1$, $\Tcal_k \leq 1$, and Lemma~\ref{lem:uk-gk-relation} that
  \begin{align*}
      &\ \(\Tcal_k \frac{\|S_k^{\mathrm{true}}\|_2^2}{2A_k} + \|c(X_k)\|_2 - \|c(X_k) + \nabla c(X_k)^TV_k\|_2 \) \nonumber \\
      &\ - \(\Tcal_k \frac{\|S_k\|_2^2}{2A_k} + \|c(X_k)\|_2 - \|c(X_k) + \nabla c(X_k)^TV_k\|_2 \) \nonumber \\
      =&\ \Tcal_k \frac{\|S_k^{\mathrm{true}}\|_2^2 - \|S_k\|_2^2}{2A_k} \nonumber \\
      =&\ \frac{\Tcal_k}{2A_k} \( (S_k^{\mathrm{true}} -S_k)^T (S_k^{\mathrm{true}} - S_k) + 2S_k^T (S_k^{\mathrm{true}} - S_k)\) \nonumber \\
      \leq&\ \frac{\Tcal_k}{2} \(\frac{\|S_k^{\mathrm{true}} - S_k\|_2^2}{A_k^2} + 2 \|S_k\|_2 \frac{\|S_k^{\mathrm{true}} - S_k\|_2}{A_k} \) \nonumber \\
      \leq&\ \frac{\Tcal_k}{2} \bigg(\kappa^2 \(\|S_k\|_2^2 + \|c(X_k)\|_2 - \|c(X_k) + \nabla c(X_k)^TV_k\|_2\) \nonumber \\
      &\qquad\ + 2\|S_k\|_2 \kappa \sqrt{\|S_k\|_2^2 + \|c(X_k)\|_2 - \|c(X_k) + \nabla c(X_k)^TV_k\|_2} \bigg) \nonumber \\
      \leq&\ \frac{\Tcal_k(2\kappa + \kappa^2)}{2}\( \frac{\|S_k\|_2^2}{A_k} + \|c(X_k)\|_2 - \|c(X_k) + \nabla c(X_k)^T V_k\|_2 \) \nonumber \\
      \leq&\ (2\kappa + \kappa^2) \( \Tcal_k \frac{ \|S_k\|_2^2}{2A_k} + \|c(X_k)\|_2 - \|c(X_k) + \nabla c(X_k)^T V_k\|_2 \),
  \end{align*}
  which after rearrangement completes the proof of part (i).

  Now suppose the condition in (ii) holds. Using the Cauchy-Schwarz inequality, inequality~\eqref{eq:grad-diff-first-v2} in the proof of Lemma~\ref{lem:gradient_error_bound}(ii), Lemma~\ref{lem:uk-gk-relation}, and $\Tcal_k \leq 1$ it follows that 
  \begin{align*}
    &\ \( \frac{\Tcal_k\|S_k^{\mathrm{true}}\|_2^2}{2A_k} + \|c(X_k)\|_2 - \|c(X_k)+\nabla c(X_k)^TV_k\|_2 \) \\
    &\ - \( \frac{\Tcal_k\|S_k\|_2^2}{2A_k} + \|c(X_k)\|_2 - \|c(X_k) + \nabla c(X_k)^TV_k\|_2 \) \\
    =&\ \Tcal_k\frac{\(\|S_k^{\mathrm{true}}\|_2^2 - \|S_k\|_2^2\)}{2A_k} \\
    =&\ \frac{\Tcal_k}{2A_k} \( 2(S_k^{\mathrm{true}})^T(S_k^{\mathrm{true}} - S_k) - \|S_k^{\mathrm{true}} - S_k\|_2^2 \) \\
    \leq&\ \frac{\Tcal_k\|S_k^{\mathrm{true}}\|_2 \|S_k^{\mathrm{true}} - S_k\|_2}{A_k} \\
    \leq&\ \Tcal_k\|S_k^{\mathrm{true}}\|_2 \|G_k - \nabla f(X_k)\|_2 \\
    \leq&\ \eta \Tcal_k\|S_k^{\mathrm{true}}\|_2 \sqrt{\frac{\|S_k^{\mathrm{true}}\|_2^2}{A_k^2} + \frac{\|c(X_k)\|_2 - \|c(X_k) + \nabla c(X_k)^TV_k\|_2}{A_k}} \\
    \leq&\ \eta \Tcal_k A_k \sqrt{\frac{\|S_k^{\mathrm{true}}\|_2^2}{A_k^2} + \frac{2(\|c(X_k)\|_2 - \|c(X_k) + \nabla c(X_k)^TV_k\|_2)}{A_k}} \\
    &\quad\ \cdot \sqrt{\frac{\|S_k^{\mathrm{true}}\|_2^2}{A_k^2} + \frac{2(\|c(X_k)\|_2 - \|c(X_k) + \nabla c(X_k)^TV_k\|_2)}{A_k}} \\
    =&\ \eta \Tcal_k A_k \( \frac{\|S_k^{\mathrm{true}}\|_2^2}{A_k^2} + \frac{2(\|c(X_k)\|_2 - \|c(X_k) + \nabla c(X_k)^TV_k\|_2)}{A_k} \) \\
    =&\ 2\eta \Tcal_k \( \frac{\|S_k^{\mathrm{true}}\|_2^2}{2A_k} + \|c(X_k)\|_2 - \|c(X_k) + \nabla c(X_k)^TV_k\|_2 \) \\
    \leq&\ 2\eta \( \Tcal_k \frac{\|S_k^{\mathrm{true}}\|_2^2}{2A_k} + \|c(X_k)\|_2 - \|c(X_k) + \nabla c(X_k)^TV_k\|_2 \),
  \end{align*}
  which after rearrangement completes the proof for part (ii).
  \qed
\end{proof}

The following lemma bounds some key differences between iterations.

\begin{lemma}\label{lem.first_lemma}
For all $k\in\N{}$, it holds that
\bsubequations
\begin{align}
| (\Phi_k^+ - \Phi_k) - (\phi(X_k^+,\Tcal_k) - \phi(X_k,\Tcal_k))|
&\leq \Tcal_k(E_k+E_k^+) \ \ \text{and} \label{phi-diff-abs} \\
W_{k+1} - W_k &\leq \phi(X_{k+1},\Tcal_k) - \phi(X_k,\Tcal_k). \label{w-diff}
\end{align}
\esubequations
\end{lemma}
\begin{proof}
  It follows from the definitions of $\Phi_k^+$ and $\Phi_k$ that
  \bequationNN
    \Phi_k^+ - \Phi_k = \phi(X_k^+,\Tcal_k) - \phi(X_k,\Tcal_k) - \Tcal_k(f(X_k^+) - F_k^+) + \Tcal_k(f(X_k) - F_k).
  \eequationNN
  After rearrangement and using the triangle inequality, one obtains~\eqref{phi-diff-abs}.

  The inequality in~\eqref{w-diff} follows using $T_{k+1} \leq \Tcal_k$ since
  \begin{align*}
    &W_{k+1} - W_k \\
    &= \Tcal_{k+1}(f(X_{k+1}) - f_{\inf} + r(X_{k+1}) - r_{\inf}) + \|c(X_{k+1})\|_2 \\
    &\phantom{bla} -(\Tcal_k(f(X_k) - f_{\inf} + r(X_k) - r_{\inf}) + \|c(X_k)\|_2) \\
    &\leq  \Tcal_k(f(X_{k+1}) - f_{\inf} + r(X_{k+1}) - r_{\inf}) + \|c(X_{k+1})\|_2 \\
    &\phantom{bla} - (\Tcal_k(f(X_k) - f_{\inf} + r(X_k) - r_{\inf}) + \|c(X_k)\|_2) \\
    &= \Tcal_k(f(X_{k+1}) + r(X_{k+1})) + \|c(X_{k+1})\|_2 - (\Tcal_k(f(X_k) + r(X_k))  + \|c(X_k)\|_2) \\
    &= \phi(X_{k+1},\Tcal_k) - \phi(X_k,\Tcal_k),
  \end{align*}
  which completes the proof of the lemma.
  \qed
\end{proof}

We now summarize key properties of the stochastic process.


\begin{lemma}
\label{lem:alg-properties}
  Suppose that $\epsilon$ satisfies~\eqref{eq:epsilon_lower_bound}.  Moreover, define $\omega:\R{2}_{\geq 0}\to \R{}_{\geq 0}$ and $h:\R{}_{>0} \to \R{}_{>0}$ for all $(a,b) \in \R{2}_{\geq 0}$ and $\alpha \in \R{}_{>0}$, respectively, by
\begin{align*}
\omega(a,b) 
&:= 2a + \tau_{-1}b \ \ \text{and} \\
h(\alpha) 
&:= \frac{\eta_1}{(\max\{\sqrt{2},\kappa_T\})^2} \min\left\{\frac{1}{2(1+\kappa)^2}, \frac{1 - 2\eta}{2}\right\} \tau_{\min}\alpha\epsilon^2.
\end{align*}
Then, the following hold.
\begin{itemize}
\item[(i)] One has that
\begin{equation}\label{bound-h-and-p}
h(\bar{\alpha}) > \frac{\omega(\epsilonbar_f, 2\epsilonbar_f)}{p-\frac{1}{2}}
\ \ \text{and} \ \ 
p \in \(\frac{1}{2}+\frac{\omega(\epsilonbar_f,2\epsilonbar_f)}{h(\bar{\alpha})}, 1\right].
\end{equation}
\item[(ii)] Conditioned on $k < T_\epsilon$, the following hold.
\begin{itemize}
\item[(a)] Conditioned on $I_k \Theta_k = 1$, one has
\begin{equation*}
  W_{k+1} \leq W_k - h(A_k) + \omega(\epsilonbar_f, E_k+E_k^+) \leq W_k - h(A_k) + \omega(\epsilonbar_f, 2\epsilonbar_f).
\end{equation*}
\item[(b)] Conditioned on $A_k \leq \bar{\alpha}$ and $I_k = 1$, one has $\Theta_k = 1$. 
\item[(c)] $W_{k+1} \leq W_k + \omega(\epsilonbar_f, E_k + E_k^+)$.  
\end{itemize}
\item[(iii)] For all $k\in\N{}$, it holds that $\P_{M_k}[I_k = 1 | \Fcal_k] \geq p$.
\end{itemize}
\end{lemma}
\begin{proof}
  Consider part (i).  It follows along with \eqref{eq:epsilon_lower_bound} and $\theta\in(0,1)$ that
\begin{equation*}
\begin{aligned}
h(\bar\alpha) 
= & \ \frac{\eta_1}{(\max\{\sqrt{2},\kappa_T\})^2} \min\left\{\frac{1}{2(1+\kappa)^2}, \frac{1 - 2\eta}{2}\right\} \tau_{\min}{\bar\alpha}\epsilon^2 \\ 
> &\ \frac{\eta_1}{2(\max\{\sqrt{2},\kappa_T\})^2} \min\left\{\frac{1}{(1+\kappa)^2}, 1 - 2\eta\right\} \tau_{\min}\alphabar \\
&\ \cdot (\max\{\sqrt{2},\kappa_T\})^2 \max\left\{(1+\kappa)^2, \frac{1}{1 - 2\eta} \right\}  \frac{(4 + 4\tau_{-1})\epsilonbar_f}{\theta\tau_{\min}\alphabar\eta_1(p - 1/2)} \\
= & \ \frac{1}{2\theta} \frac{(4 + 4\tau_{-1})\epsilonbar_f}{p - 1/2}
> \frac{(2 + 2\tau_{-1})\epsilonbar_f}{p - 1/2}
= \frac{\omega(\epsilonbar_f, 2\epsilonbar_f)}{p-1/2},
\end{aligned}
\end{equation*}
as claimed.  After rearrangement, this inequality along with the bounds in \eqref{def:p} together imply that $p$ is on the stated interval in part~(i).

Now consider part (ii) and condition on the event that $k < T_\epsilon$.  Let us begin by proving part (ii)(a), for which we condition on the event that $I_k\Theta_k = 1$. Thus, iteration $k$ is true, so the inequalities in Definition~\ref{def:true} hold. We consider two cases.  
\begin{itemize}
  \item Suppose that $\|G_k - \nabla f(X_k)\|_2 \leq \kappa \sqrt{\|S_k\|_2^2 + \|c(X_k)\|_2 - \|c(X_k) + \nabla c(X_k)^T V_k\|_2}$. Thus, by Lemma~\ref{lem:sk-difference}, the conclusion of part (i) of that lemma holds. Combining this conclusion with the fact that iteration $k$ is also successful and $\Tcal_k \leq 1$, 
  \begin{align*}
    &\ \Phi_k^+ - \Phi_k - 2\epsilonbar_f \\
    \leq & \ -\eta_1 \(\Tcal_k\frac{\|S_k\|_2^2}{2A_k} + (\|c(X_k)\|_2 - \|c(X_k) + \nabla c(X_k)^TV_k\|_2)\) \\
    \leq & \ -\frac{\eta_1}{(1+\kappa)^2} \(\Tcal_k\frac{\|S_k^{\mathrm{true}}\|_2^2}{2A_k} + (\|c(X_k)\|_2 - \|c(X_k) + \nabla c(X_k)^TV_k\|_2)\) \\
    \leq & \ -\frac{\eta_1}{(1+\kappa)^2} \Tcal_k\(\frac{\|S_k^{\mathrm{true}}\|_2^2}{2A_k} + (\|c(X_k)\|_2 - \|c(X_k) + \nabla c(X_k)^TV_k\|_2)\) \\
    = & \ -\frac{\eta_1}{(1+\kappa)^2} \Tcal_k A_k\(\frac{\|S_k^{\mathrm{true}}\|_2^2}{2A_k^2} + \frac{\|c(X_k)\|_2 - \|c(X_k) + \nabla c(X_k)^TV_k\|_2}{A_k}\) \\
    \leq & \ -\frac{\eta_1}{2(1+\kappa)^2} \Tcal_k A_k \(\frac{\|S_k^{\mathrm{true}}\|_2^2}{A_k^2} + \frac{\|c(X_k)\|_2 - \|c(X_k) + \nabla c(X_k)^TV_k\|_2}{A_k}\).
  \end{align*}
  \item Suppose that $\|G_k - \nabla f(X_k)\|_2 \leq \epsilon_g$. Thus, by Lemma~\ref{lem:sk-difference}, the conclusion of part (ii) of that lemma holds. Combining this conclusion with the fact that iteration $k$ is also successful and $\Tcal_k \leq 1$,
  \begin{align*}
    &\ \Phi_k^+ - \Phi_k - 2\epsilonbar_f \\ 
    \leq & \ -\eta_1 \(\Tcal_k\frac{\|S_k\|_2^2}{2A_k} + (\|c(X_k)\|_2 - \|c(X_k) + \nabla c(X_k)^TV_k\|_2)\) \\
    \leq & \ - \eta_1(1 - 2\eta)\(\Tcal_k \frac{\|S_k^{\mathrm{true}}\|_2^2}{2A_k} + (\|c(X_k)\|_2 - \|c(X_k) + \nabla c(X_k)^TV_k\|_2)\) \\
    \leq & \ - \eta_1(1 - 2\eta)\Tcal_k \(\frac{\|S_k^{\mathrm{true}}\|_2^2}{2A_k} + (\|c(X_k)\|_2 - \|c(X_k) + \nabla c(X_k)^TV_k\|_2)\) \\
    = & \ -\eta_1(1 - 2\eta)\Tcal_k A_k \(\frac{\|S_k^{\mathrm{true}}\|_2^2}{2A_k^2} + \frac{\|c(X_k)\|_2 - \|c(X_k) + \nabla c(X_k)^TV_k\|_2}{A_k}\) \\
    \leq & \ -\eta_1\frac{1 - 2\eta}{2} \Tcal_k A_k \(\frac{\|S_k^{\mathrm{true}}\|_2^2}{A_k^2} + \frac{\|c(X_k)\|_2 - \|c(X_k) + \nabla c(X_k)^TV_k\|_2}{A_k}\).
  \end{align*}
  \end{itemize}
  Combining these inequalities, Lemma~\ref{lem:merit-bounded}, $A_k \leq 1$, and $\|S_k^{\mathrm{true}}\|_2 \geq \|\min\{X_k, -Z_k^{\mathrm{true}}\}\|_2$ (see \cite[Lemma 4.6]{curtis2025proximal}) one finds with $\chi := \min\left\{\frac{1}{2(1+\kappa)^2}, \frac{1 - 2\eta}{2}\right\}$ that
  \begin{align*}
    &\ \Phi_k^+ - \Phi_k - 2\epsilonbar_f \\ 
    \leq&\ -\eta_1 \chi \Tcal_k A_k \(\frac{\|S_k^{\mathrm{true}}\|_2^2}{A_k^2} + \frac{\|c(X_k)\|_2 - \|c(X_k) + \nabla c(X_k)^TV_k\|_2}{A_k}\) \\
    \leq&\ -\eta_1 \chi \Tcal_k A_k \bigg(\frac{\|S_k^{\mathrm{true}}\|_2^2}{2A_k^2} + \frac{\|\min\{X_k, -Z_k^{\mathrm{true}}\}\|_2^2}{2A_k^2} \\
    &\qquad\qquad\qquad\ + \frac{\|c(X_k)\|_2 - \|c(X_k) + \nabla c(X_k)^TV_k\|_2}{A_k}\bigg) \\
    \leq&\ -\eta_1 \chi \tau_{\min}A_k \bigg(\frac{\|S_k^{\mathrm{true}}\|_2^2}{2A_k^2} + \frac{\|\min\{X_k, -Z_k^{\mathrm{true}}\}\|_2^2}{2} \\
    &\qquad\qquad\qquad\ + \frac{\|c(X_k)\|_2 - \|c(X_k) + \nabla c(X_k)^TV_k\|_2}{A_k}\bigg) \\
    \leq&\ -\frac{\eta_1}{(\max\{\sqrt{2},\kappa_T\})^2} \chi \tau_{\min}A_k\epsilon^2 \\
    =&\ -h(A_k).
  \end{align*}
  Hence, it follows with \eqref{phi-diff-abs}, nonnegativity of $E_k$ and $E_k^+$, and $\Tcal_k \leq \tau_{-1}$ that
  \begin{align*}
    \phi(X_k^+,\Tcal_k) - \phi(X_k,\Tcal_k) &\leq \Phi_k^+ - \Phi_k + \Tcal_k(E_k + E_k^+) \\
    &\leq -h(A_k) + 2 \epsilonbar_f + \tau_{-1}(E_k + E_k^+) \\
    &= -h(A_k) + \omega(\epsilonbar_f, E_k + E_k^+).
  \end{align*}
  Combining this inequality, \eqref{w-diff}, $\Theta_k = 1$, the fact that $\omega(\cdot,\cdot)$ is an increasing function in its second argument, and $E_k + E_k^+ \leq 2 \epsilonbar_f$ give
  \begin{align*}
    W_{k+1} - W_k 
    &\leq \phi(X_{k+1},\Tcal_k) - \phi(X_k,\Tcal_k) = \phi(X_k^+,\Tcal_k) - \phi(X_k,\Tcal_k) \\
    &\leq -h(A_k) + \omega(\epsilonbar_f,E_k+E_k^+) \leq -h(A_k) + \omega(\epsilonbar_f,2\epsilonbar_f),
  \end{align*}
  which completes the proof of part~(ii)(a).


  Now consider part (ii)(b), for which we condition on the event that $A_k \leq \bar\alpha$ and $I_k = 1$. Thus, the inequalities in Definition~\ref{def:true} hold. Let us consider two cases.
  \begin{itemize}
    \item Suppose that $\|G_k - \nabla f(X_k)\|_2 \leq \kappa\sqrt{\|S_k\|_2^2 + \|c(X_k)\|_2 - \|c(X_k) + \nabla c(X_k)^TV_k\|_2}$.  Let us first observe from $\tau_{\min} \leq \Tcal_k$ and $A_k \leq \alphabar \leq \omega_1^{-1}$ that 
    \begin{align*}
      & \ 1 - \(\frac{\Tcal_kL_g + L_J}{\Tcal_k}\) A_k - \(\frac{\kappa^2}{2\eta_1}+ 2\eta_1\) A_k \\
      \geq &\ 1 - \(\frac{\tau_{\min}L_g + L_J}{\tau_{\min}} + \frac{\kappa^2}{2\eta_1}+ 2\eta_1\) A_k \\
      \geq &\ 1 - (1-\eta_1)\(\frac{\tau_{\min}(L_g+2\eta_1) + L_J}{\tau_{\min}} + \frac{\kappa^2}{2\eta_1}\)\(\frac{\tau_{\min}(L_g + 2\eta_1) + L_J}{\tau_{\min}} + \frac{\kappa^2}{2\eta_1}\)^{-1} \\
      =&\  \eta_1. 
    \end{align*}
    Combining this inequality with Lemma~\ref{lem:merit-reduction-new}(i) gives
    \begin{align*}
      &\ \phi(X_k + S_k,\Tcal_k) - \phi(X_k,\Tcal_k) \\
      \leq& \ -\frac{\Tcal_k}{2 A_k}\|S_k\|_2^2 + \frac{\Tcal_k L_g + L_J}{2} \|S_k\|_2^2 + \(\frac{\kappa^2}{4\eta_1}+ \eta_1\) A_k \Tcal_k \frac{\|S_k\|_2^2}{A_k} \\
      &\ - \eta_1 (\|c(X_k)\|_2 - \|c(X_k) + \nabla c(X_k)^TV_k\|_2) \\
      = & \ - \(1 - \(\frac{\Tcal_k L_g + L_J}{\Tcal_k}\) A_k - \(\frac{\kappa^2}{2\eta_1} + 2\eta_1\) A_k\) \frac{\Tcal_k}{2A_k} \|S_k\|_2^2 \\
      &\ - \eta_1 (\|c(X_k)\|_2 - \|c(X_k) + \nabla c(X_k)^TV_k\|_2) \\
      \leq & \ - \eta_1\frac{\Tcal_k}{2A_k} \|S_k\|_2^2 - \eta_1 (\|c(X_k)\|_2 - \|c(X_k) + \nabla c(X_k)^TV_k\|_2) \\
      = & \ -\eta_1 \(\frac{\Tcal_k}{2A_k} \|S_k\|_2^2 +(\|c(X_k)\|_2 - \|c(X_k) + \nabla c(X_k)^TV_k\|_2)\). 
    \end{align*}
    \item Suppose that $\|G_k - \nabla f(X_k)\|_2 \leq \epsilon_g$. From $\tau_{\min} \leq \Tcal_k$ and $A_k \leq \alphabar \leq \omega_2^{-1}$,
    \begin{align*}
      &\ 1 - \(\frac{\eta^2}{2\eta_1(1-2\eta)^2} + 2\eta_1\) - \frac{(\Tcal_k L_g + L_J) A_k}{\Tcal_k} \\
      \geq&\ 1 - \(\frac{\eta^2}{2\eta_1(1-2\eta)^2} + 2\eta_1\) - \frac{(\tau_{\min} L_g + L_J)\omega_2^{-1}}{\tau_{\min}} \\
      =&\ 1 - \frac{\eta^2}{2\eta_1(1-2\eta)^2} - 2\eta_1 - \(1 - \frac{\eta^2}{2\eta_1(1-2\eta)^2} - 3\eta_1\) \\
      =&\ \eta_1.
    \end{align*}
    Combining this inequality with Lemma~\ref{lem:merit-reduction-new}(ii) gives
    \begin{align*}
      &\ \phi(X_k + S_k,\Tcal_k) - \phi(X_k,\Tcal_k) \\
      \leq & \ -\frac{\Tcal_k}{2A_k} \|S_k\|_2^2 + \frac{\Tcal_k L_g + L_J}{2} \|S_k\|_2^2 + \(\frac{\eta^2}{4\eta_1(1-2\eta)^2} + \eta_1\) \Tcal_k\frac{\|S_k\|_2^2}{A_k} \\
      &\ - \eta_1 (\|c(X_k)\|_2 - \|c(X_k) + \nabla c(X_k)^TV_k\|_2) \\
      = & \ - \(1 - \(\frac{\eta^2}{2\eta_1(1-2\eta)^2} + 2\eta_1\) - \(\frac{\Tcal_kL_g + L_J}{\Tcal_k}\) A_k \) \frac{\Tcal_k}{2A_k} \|S_k\|_2^2 \\
      &\ - \eta_1(\|c(X_k)\|_2 - \|c(X_k) + \nabla c(X_k)^TV_k\|_2) \\
      \leq & \ -\eta_1 \frac{\Tcal_k}{2A_k} \|S_k\|_2^2 - \eta_1 (\|c(X_k)\|_2 - \|c(X_k) + \nabla c(X_k)^TV_k\|_2) \\
      = & \ -\eta_1 \(\frac{\Tcal_k}{2A_k} \|S_k\|_2^2 +(\|c(X_k)\|_2 - \|c(X_k) + \nabla c(X_k)^TV_k\|_2)\). 
    \end{align*}
  \end{itemize}
  It follows from the previous two cases, \eqref{phi-diff-abs}, $E_k + E_k^+ \leq 2 \epsilonbar_f$, and $\Tcal_k \leq 1$ that
  \begin{align*}
    \Phi_k^+ - \Phi_k
    &\leq \phi(X_k^+,\Tcal_k) - \phi(X_k,\Tcal_k) + \Tcal_k(E_k + E_k^+) \leq \phi(X_k^+,\Tcal_k) - \phi(X_k,\Tcal_k) + 2 \epsilonbar_f \\
    &\leq -\eta_1 \(\frac{\Tcal_k}{2A_k} \|S_k\|_2^2 +(\|c(X_k)\|_2 - \|c(X_k) + \nabla c(X_k)^TV_k\|_2)\)  + 2\epsilonbar_f,
  \end{align*}
which implies that $\Theta_k = 1$, thus completing the proof of part (ii)(b).

  Now consider part (ii)(c).  Note that $W_{k+1} \leq W_k$ if the $k$th iteration is not successful since $\Tcal_{k+1} \leq \Tcal_k$ and $X_{k+1} = X_k$, in which case the inequality in part (ii)(c) holds trivially. Thus, for the remainder of the proof of this part, let us condition on the event that $\Theta_k = 1$.  In this case, one finds that
  \begin{equation*}
    \Phi_k^+ - \Phi_k \leq -\eta_1 \(\frac{\Tcal_k}{2A_k} \|S_k\|_2^2 + (\|c(X_k)\|_2 - \|c(X_k) + \nabla c(X_k)^TV_k\|_2)\) + 2\epsilonbar_f \leq 2\epsilonbar_f.
  \end{equation*}
  It follows from this inequality, $\Theta_k = 1$, \eqref{w-diff}, \eqref{phi-diff-abs}, and $\Tcal_k \leq \tau_{-1}$ that 
  \begin{align*}
    W_{k+1} - W_k &\leq \phi(X_k^+,\Tcal_k) - \phi(X_k,\Tcal_k) \leq \Phi_k^+ - \Phi_k + \Tcal_k(E_k+E_k^+) \\
    &\leq 2\epsilonbar_f + \tau_{-1}(E_k+E_k^+) = \omega(\epsilonbar_f,E_k+E_k^+),
  \end{align*}
  which completes the proof of part (ii)(c).


  Finally, consider part (iii). Letting $1\{A\}$ denote the indicator function for the event $A$, let us define for convenience the random variables 
  \begin{align*}
    I_{k,1} &:= 1\bigg\{\|G_k - \nabla f(X_k)\| \\
    &\qquad\quad\ \leq \max\left\{\epsilon_g, \kappa \sqrt{\|S_k\|_2^2 + \|c(X_k)\|_2 - \|c(X_k) + \nabla c(X_k)^TV_k\|_2} \right\} \bigg\} \\
    \text{and}\ \ 
    I_{k,2} &:= 1\left\{ E_k + E_k^+ \leq 2\epsilonbar_f \right\}.
  \end{align*}
  It follows from these definitions and Assumption~\ref{ass:stochastic_first_order_oracle} that
  \begin{align}
    \P_{M_k}[I_k=0 | \Fcal_k] &= \P_{M_k}[I_{k,1}=0 \text{ or } I_{k,2} = 0 | \Fcal_k] \nonumber \\
    &\leq \P_{M_k}[I_{k,1} = 0  |  \Fcal_k] + \P_{M_k}[I_{k,2} = 0  |  \Fcal_k] \nonumber \\
    &\leq \delta + \P_{M_k}[I_{k,2} = 0 | \Fcal_k]. \label{prob-to-bound}
  \end{align}
  We now bound the probability on the right-hand side of~\eqref{prob-to-bound}.  It follows from~\eqref{eq:subexp-E}, \eqref{eq:subexp-E+}, $M_k = (\Xi_{k,g},\Xi_{k,f},\Xi^+_{k,f     })$, linearity of the expectation operator, and the Cauchy-Schwarz inequality for expectations that
  \begin{align*}
    \E_{M_k}[E_k+E_k^+ | \Fcal_k] \leq&\ 2\epsilon_f
    \quad \text{and} \\
    \E_{M_k}\!\left[
      \exp\!\left\{
        \lambda\bigl(E_k+E_k^+ - \E_{M_k}[E_k+E_k^+|\Fcal_k]\bigr)
      \right\} | \Fcal_k
    \right]
    \leq&\ \exp\(\frac{\lambda^2 (2\nu)^2}{2}\) \\
    &\ \text{for all} \ \ \lambda \in \left[0, \frac{1}{2b}\right].
  \end{align*}
  Using a one-sided Bernstein inequality, it follows for all $t \geq 0$ that
  \begin{equation}\label{eq:one-sided-bernstein}
    \P_{M_k}[E_k + E^+_k - \E_{M_k}[E_k + E_k^+ | \Fcal_k] \geq t  |  \Fcal_k] \leq \exp\(-\min\left\{\frac{t^2}{2(2\nu)^2}, \frac{t}{2 \cdot 2b}\right\}\).
  \end{equation}
  Note that
  \begin{align}
    \P_{M_k}[I_{k,2} = 0 | \Fcal_k]
    &= \P_{M_k} \bigg[E_k + E_k^+ - \E_{M_k}[E_k + E_k^+ | \Fcal_k] \nonumber \\
    &\qquad\qquad > 2\epsilonbar_f - \E_{M_k}[E_k + E_k^+ | \Fcal_k] \Big| \Fcal_k \bigg]. \label{Prob-for-Ik2}
  \end{align}
  Next, let us define $\tbar := 2\epsilonbar_f - \E_{M_k}[E_k+E_k^+\, | \, \Fcal_k]$, and note from linearity of the expectation operator, \eqref{eq:subexp-E}, \eqref{eq:subexp-E+}, and Assumption~\ref{ass:p} that $\tbar \geq  2u_f \geq 0$. Combining this inequality with~\eqref{Prob-for-Ik2} and~\eqref{eq:one-sided-bernstein}, one obtains
  \begin{align*}
    \P_{M_k}[I_{k,2} = 0 | \Fcal_k]
    &= \P_{M_k} [E_k + E_k^+ - \E_{M_k}[E_k + E_k^+ | \Fcal_k] ] > \tbar | \Fcal_k ] \\
    &\leq \exp\(-\min\left\{ \frac{\tbar^2}{2(2\nu)^2}, \frac{\tbar}{2\cdot 2b}
\right\}\) \\
&\leq \exp\(-\min\left\{
  \frac{4u_f^2}{2(2\nu)^2},
  \frac{2u_f}{2\cdot 2b}
\right\}\) 
= \exp\(-\min\left\{
  \frac{u_f^2}{2\nu^2},
  \frac{u_f}{2b}
\right\}\). 
%
  \end{align*}
  Combining this inequality with~\eqref{prob-to-bound} gives 
  \bequationNN
    \P_{M_k}[I_k = 1 | \Fcal_k] = 1 - \P_{M_k}[I_k = 0 | \Fcal_k] \geq 1 - \delta - \exp\(-\min\left\{\frac{u_f^2}{2\nu^2}, \frac{u_f}{2b}\right\}\) = p,
  \eequationNN
  which completes the proof.
  \qed
\end{proof}





Our next result gives a concentration bound related to the sum of $\{I_k\}$.

\begin{lemma}\label{lem:true-total-prob}
  Let $\hat{p} \in [0, p)$.  Then, it holds that
  \bequationNN
    \P\left[\sum_{k=0}^{t-1} I_k \leq \hat{p} t \right] \leq \exp\(-\frac{(p - \hat{p})^2}{2p^2}t\) \ \ \text{for all $t \in \N{}$ with $t \geq 1$.}
  \eequationNN
\end{lemma}
\begin{proof}
  Consider arbitrary $\hat{p} \in [0, p)$.  For all $t \in \N{}$ with $t \geq 1$ define the $\Fcal_t$-measurable random variable $\Ybar_t := pt - \sum_{k=0}^{t-1}I_k$ and $\Ybar_0 := 0$.  It follows that
  \begin{equation}\label{Y-recursion}
    \Ybar_{t+1} =  p(t+1) - \sum_{k=0}^{t}I_k =  pt + p - I_t - \sum_{k=0}^{t-1}I_k = \Ybar_t + p - I_t.
  \end{equation}
  Combining~\eqref{Y-recursion}, linearity of expectation, and Lemma~\ref{lem:alg-properties}(iii) gives
  \begin{align*}
    \E_{M_t}[\Ybar_{t+1} |  \Fcal_t]
    &= \E_{M_t}[\Ybar_t + p - I_t | \Fcal_t] = \Ybar_t + p - \E_{M_t}[I_t | \Fcal_t] \\ 
    &= \Ybar_t + p - \P_{M_t}[I_t=1 | \Fcal_t] \leq \Ybar_t + p - p
    = \Ybar_t,
  \end{align*}
  from which it follows that $\{\Ybar_t\}$ is a supermartingale. Next, it follows from \eqref{Y-recursion}, the fact that $I_k \in \{0,1\}$ for all $k\in\N{}$, and $p > 1/2$ (see~\eqref{def:p}) that
  \begin{align*}
    |\Ybar_t - \Ybar_{t-1}| = |p - I_{t-1}| \leq \max\{p, 1-p\} = p.
  \end{align*}
  Thus, each increment of the supermartingale $\{\Ybar_t\}$ is bounded in absolute value by $p$. Hence, the Azuma--Hoeffding inequality~\cite{azuma1967weighted} gives
  \begin{align*}
    \P\left[\sum_{k=0}^{t-1} I_k \leq \hat{p} t \right] &= \P\left[pt - \sum_{k=0}^{t-1} I_k \geq (p-\hat{p})t \right] 
    = \P[\Ybar_t \geq (p-\phat)t] \\
    &\leq \exp \(-\frac{(p-\hat{p})^2t^2}{2tp^2}\) = \exp \(-\frac{(p-\hat{p})^2}{2p^2}t\),
  \end{align*}
  which completes the proof.
  \qed
\end{proof}

Next, we show that the number of successful iterations with a ``large'' proximal parameter (Definition~\ref{def:largestep}) is at least as large as the number of unsuccessful iterations with a ``large'' proximal parameter minus a constant. The result and its proof are similar to~\cite[Lemma~3.2]{jin2024high} with a small modification to handle the upper bound enforced on the proximal parameter in Algorithm~\ref{alg:main}, which was not used in~\cite{jin2024high}.

\begin{lemma}\label{lem:large}
With $\ibar$ from Definition~\ref{def:stoppingtime}, it holds for all $t\in\N{}\setminus\{0\}$ that
$$
\sum_{k=0}^{t-1} L_k \Theta_k 
\geq \sum_{k=0}^{t-1} 1\{A_k< \alpha_0\}L_k \Theta_k
\geq \sum_{k=0}^{t-1} L_k (1 - \Theta_k) - (\ibar+1).
$$
\end{lemma}
\begin{proof}
  The first inequality is trivial.  Next, since a large proximal parameter in iteration $k$ has $A_k$ and $A_{k+1}$ bounded below by $\bar{\alpha}$, the number of large proximal parameters for which $A_k$ gets decreased is bounded by the number of large proximal parameters with $A_k < \alpha_0$ for which $A_k$ gets increased, plus $\ibar+1$, which is the number of unsuccessful steps it takes to decrease the proximal parameter from $\alpha_0$ to $\bar{\alpha}$. This establishes the second inequality above.
%
  \qed
\end{proof}

The next result is the same as~\cite[Corollary 3.3]{jin2024high}.

\begin{lemma}\label{lem:sum-LTheta}
  For all $t\in\N{}\setminus\{0\}$, it holds that
$
\sum_{k=0}^{t-1} L_k \Theta_k \geq \frac{1}{2} \( \sum_{k=0}^{t-1} L_k - (\ibar+1)\).
$
\end{lemma}
\begin{proof}
Use the identity $L_k(1 - \Theta_k) = L_k - L_k\Theta_k$ in Lemma~\ref{lem:large} and rearrange.
  \qed
\end{proof}

The next result is the same as~\cite[Lemma~3.4]{jin2024high}, with additional details of the proof provided here for enhanced clarity.

\begin{lemma}\label{lem:diff-true-iters}
  Conditioned on $t < T_\epsilon$, we have $\sum_{k=0}^{t-1} (1 - L_k) I_k \leq \sum_{k=0}^{t-1} (1 - L_k)(1 - I_k)$.
\end{lemma}
\begin{proof}
  It follows from Lemma~\ref{lem:alg-properties}(ii)(b) and the fact that $A_k \leq \alphabar$ when $L_k = 0$ that
  \begin{equation}\label{eq:bd-small-steps-1}
    \sum_{k=0}^{t-1} (1 - L_k) I_k \leq \sum_{k=0}^{t-1} (1 - L_k) \Theta_k \ \ \text{and} \ \ \sum_{k=0}^{t-1} (1 - L_k)(1 - \Theta_k) \leq \sum_{k=0}^{t-1} (1 - L_k)(1 - I_k).
  \end{equation}
  Also, since the number of small proximal parameters where $A_k$ is increased is bounded by the number of small proximal parameters where $A_k$ is decreased,
  \begin{equation}\label{eq:bd-small-steps-2}
    \sum_{k=0}^{t-1} (1 - L_k) \Theta_k \leq \sum_{k=0}^{t-1} (1 - L_k)(1 - \Theta_k).
  \end{equation}
  The claimed result follows by combining~\eqref{eq:bd-small-steps-1} and~\eqref{eq:bd-small-steps-2}.
  \qed
\end{proof}

We can now prove that the probability that a certain event occurs is zero.

\begin{lemma}\label{lem:zero-prob}
  Let $\hat{p} \in (\thalf, 1]$.  Then, for all $t\in\N{}\setminus\{0\}$, it holds that
  \bequationNN
    \P{}\left[ T_{\epsilon} > t,\ \sum_{k=0}^{t-1} I_k \geq \hat{p}t,\ \text{and}\ \sum_{k=0}^{t-1} L_k\Theta_kI_k < \(\hat{p} - \frac{1}{2}\) t - \frac{(\ibar+1)}{2}\right] = 0.
  \eequationNN
\end{lemma}
\begin{proof}
The proof follows exactly as the proof for~\cite[Lemma 3.5]{jin2024high}, which uses~\cite[Corollary 3.3]{jin2024high}) (equivalent to our Lemma~\ref{lem:sum-LTheta}), \cite[Lemma 3.4]{jin2024high} (equivalent to our Lemma~\ref{lem:diff-true-iters}), and the notation $U_k$ (analogous to our notation $L_k$). \qed
\end{proof}

We now prove a worst-case iteration complexity result.

\begin{theorem}\label{thm:main}
  Suppose that Assumptions~\ref{ass:general}--\ref{ass:kappaG} hold, let $T_\epsilon$ be defined as in Definition~\ref{def:stoppingtime}. Then, the set of $(s,\phat,t) \in \R{} \times \R{} \times \N{}$ satisfying
  \bsubequations\label{eq:X_knows_better}
  \begin{align}
s &\in \left(0, \(p - \tfrac{1}{2}\) h(\bar{\alpha}) - \omega(\epsilonbar_f, 2\epsilonbar_f)\), \label{s-in}\\
\hat{p} &\in \( 
\frac{1}{2} + \frac{\omega(\epsilonbar_f, 2\epsilonbar_f) + s}{h(\bar{\alpha})}, p 
\), \ \ \text{and} \label{phat-in}\\
t &\geq \that := 
\frac{\Wbar_0/h(\bar{\alpha}) + (\ibar+1)/2}
{
\hat{p} - 
\big(\frac{1}{2} +
\frac{\omega(\epsilonbar_f, 2\epsilonbar_f) + s}{h(\bar{\alpha})}\big)}
\label{t-geq}
\end{align}
  \esubequations
is nonempty, where $\Wbar_0 := \tau_{-1}\big(f(x_0)-f_{\inf}+r(x_0)-r_{\inf}\big) + \|c(x_0)\|_2 \geq W_0$. Moreover, for all $(s,\phat,t)$ satisfying~\eqref{eq:X_knows_better}, it holds that
\begin{equation}\label{eq:subexponential}
\P[T_{\epsilon} \leq t] \geq 
1 - \exp\(
-\frac{(p - \hat{p})^2}{2p^2}  t
\)
- \exp\(
-\min\left\{
\frac{s^2 t}{2(2\tau_{-1}\nu)^2}, 
\frac{s t}{2(2\tau_{-1}b)}
\right\}
\),
\end{equation}
where the right-hand side of~\eqref{eq:subexponential} is positive for all $t$ sufficiently large. Thus, Lemma~\ref{lem:kkt-result} gives for all such $t$ that, with probability at least the right-hand side of~\eqref{eq:subexponential},  Algorithm~\ref{alg:main} computes an iteration $\bar{k} \leq t$ so that $x_{\kbar}$ is an $\epsilon$-KKT point for problem~\eqref{prob:main}.
\end{theorem}
\begin{proof}
  The interval in~\eqref{s-in} is nonempty as a consequence of~\eqref{bound-h-and-p}.  Next, let $s$ be any real number satisfying~\eqref{s-in}, which implies that $0 < s <  \(p - \tfrac{1}{2}\) h(\bar{\alpha}) - \omega(\epsilonbar_f, 2\epsilonbar_f)$. After rearranging the second inequality, one obtains $\half + \tfrac{\omega(\epsilonbar_f,2\epsilonbar_f)+s}{h(\alphabar)} < p$ so that the interval in~\eqref{phat-in} is nonempty.  Finally, let $\phat$ be any real number satisfying~\eqref{phat-in}. The denominator in~\eqref{t-geq} is positive and finite since $\phat$ satisfies~\eqref{phat-in}, and the numerator is nonnegative and finite by the definitions of $\Wbar_0$, $h(\cdot)$, $\alphabar$, and $\ibar$. Therefore, the right-hand side of~\eqref{t-geq} is nonnegative and finite, meaning that all $t$ sufficiently large satisfy~\eqref{t-geq}, which completes the proof of the first claim.

  Next, let $(s,\phat,t)$ satisfy~\eqref{eq:X_knows_better}. By the law of total probability,
  \begin{align*}
    \P[T_{\epsilon} > t]
    =&\ \P\left[ \underbrace{T_\epsilon > t \ \text{and} \ \frac{1}{t}\sum_{k=0}^{t-1} \omega(\epsilonbar_f, E_k+E_k^+) > \omega(\epsilonbar_f,2\epsilonbar_f)+s}_{\Acal} \right] \\
    &\ + \P\left[\underbrace{T_\epsilon > t \ \text{and} \ \frac{1}{t}\sum_{k=0}^{t-1} \omega(\epsilonbar_f,E_k+E_k^+) \leq \omega(\epsilonbar_f,2\epsilonbar_f)+s}_{\Bcal} \right].
  \end{align*}
  Let us now bound $\P[\Acal]$ and $\P{}[\Bcal]$ in order to prove~\eqref{eq:subexponential}.

  First, consider $\P[\Acal]$.  For any $k \in \N{}$, Assumption~\ref{ass:zeroth_order_oracle}, Assumption~\ref{ass:p} (specifically, $u_f \geq 0$), \eqref{eq:subexp-E+}, $M_k = (\Xi_{k,g},\Xi_{k,f},\Xi^+_{k,f})$, and linearity of expectation show that 
  \begin{equation}\label{exp-bounds-complexity}
    \begin{aligned}
      &\ \E_{M_k}[\tau_{-1}(E_k+E_k^+) | \Fcal_k] 
      \leq 2\tau_{-1}\epsilon_f \leq 2\tau_{-1}\epsilonbar_f \ \ \text{and} \\
      &\ \E_{M_k} \left[ \exp \left\{ \lambda\bigl(\tau_{-1}(E_k+E_k^+) - \E_{M_k}[\tau_{-1}(E_k+E_k^+)|\Fcal_k]\bigr)
      \right\} | \Fcal_k \right] \\
      \leq&\ \exp\(\frac{\lambda^2 (2\tau_{-1}\nu)^2}{2}\)\ \ \text{for all}\ \ \lambda \in \left[0, \frac{1}{2\tau_{-1}b}\right].
    \end{aligned}
  \end{equation}
  Next, \eqref{exp-bounds-complexity} and letting $\Ecal_k := \tau_{-1}(E_k+ E_k^+) - \E_{M_k}[\tau_{-1}(E_k+E_k^+) \, | \, \Fcal_k]$ gives 
\begin{equation}\label{eq:bound-A-I}
\begin{aligned}
\P[\Acal]
& \leq
\P{}\left[
\frac{1}{t}\sum_{k=0}^{t-1}
\omega(\epsilonbar_f,E_k+E_k^+)
> \omega(\epsilonbar_f,2\epsilonbar_f) + s
\right] \\
&= \P\left[
\frac{1}{t}\sum_{k=0}^{t-1}
\(2\epsilonbar_f + \tau_{-1}(E_k+E_k^+)\)
> 2\epsilonbar_f + 2\tau_{-1}\epsilonbar_f + s
\right] \\
& = \P\left[
\frac{1}{t}\sum_{k=0}^{t-1} ( \tau_{-1}(E_k + E_k^+) - 2 \tau_{-1}\epsilonbar_f ) > s
\right] \\
&\leq \P\left[
\frac{1}{t}\sum_{k=0}^{t-1} ( \tau_{-1}(E_k + E_k^+) - 
\E_{M_k}[\tau_{-1}(E_k+E_k^+) |  \Fcal_k] )  >  s
\right] \\ 
&= \P\left[\frac{1}{t}\sum_{k=0}^{t-1} \Ecal_k > s \right]
\leq \P\left[\sum_{k=0}^{t-1} \Ecal_k \geq st \right].
\end{aligned}
\end{equation}
Also, it follows from the tower property and~\eqref{exp-bounds-complexity} that 
\begin{align}
\E\!\left[\mathrm{exp}\left(\lambda\sum_{k=0}^{t-1} \Ecal_k\right)\right]  
&= \E\!\left[\E_{M_{t-1}}\left[ \mathrm{exp}\left(\lambda\sum_{k=0}^{t-2} \Ecal_k\right) \mathrm{exp}\left(\lambda\Ecal_{t-1}\right) \Big| \Fcal_{t-1} \right]\right] \label{peel-off-one} \\
&\leq \mathrm{exp}\left( \frac{\lambda^2(2\tau_{-1}\nu)^2}{2}\right) \E\left[  \mathrm{exp}\left(\lambda\sum_{k=0}^{t-2} \Ecal_k\right)  \right] \ \text{for all $\lambda \in \left[0, \frac{1}{2\tau_{-1}b}\right]$.} \nonumber
\end{align}
We can now use the Markov inequality and an iterative application of~\eqref{peel-off-one} to get
\begin{align*}
\P\left[\sum_{k=0}^{t-1} \Ecal_k \geq st \right]
&\leq \P\left[\mathrm{exp}\(\lambda\sum_{k=0}^{t-1}\Ecal_k\) \geq \mathrm{exp}(\lambda st) \right] 
\leq \mathrm{exp}(-\lambda st) \E\!\left[\mathrm{exp}\left(\lambda\sum_{k=0}^{t-1} \Ecal_k\right)\right] \\ 
&\leq \mathrm{exp}(-\lambda st) \mathrm{exp}\left( \frac{t\lambda^2(2\tau_{-1}\nu)^2}{2}\right) \\
&= \mathrm{exp}\(-\lambda st + \frac{t(2\tau_{-1}\nu)^2\lambda^2}{2}\)\ \ \text{for all $\lambda \in \left[0, \frac{1}{2\tau_{-1}b}\right]$}.
\end{align*}
We now combine this result with~\eqref{eq:bound-A-I} to obtain
\begin{equation}\label{ProbA-bound}
\P[\Acal] 
\leq \mathrm{exp}\(-\lambda st + \frac{t(2\tau_{-1}\nu)^2\lambda^2}{2}\) =: h_\Acal(\lambda)
\ \ \text{for all $\lambda \in \left[0, \frac{1}{2\tau_{-1}b}\right]$.}
\end{equation}
We now minimize $h_\Acal(\cdot)$ over all  $\lambda$ in~\eqref{ProbA-bound}. The \emph{unconstrained} minimizer of $h_\Acal$ is
\begin{align*}
\lambda^*_u := \frac{s}{(2\tau_{-1}\nu)^2}.
\end{align*}
\textbf{Case 1.} If $\lambda^*_u < \frac{1}{2\tau_{-1}b}$, i.e., if $s < \frac{(2\tau_{-1}\nu)^2}{2\tau_{-1}b}$, then the minimizer of $h_\Acal$ is $\lambda^* = \lambda^*_u$ with $h_\Acal(\lambda^*) = \mathrm{exp}\(-\frac{s^2 t}{2(2\tau_{-1}\nu)^2}\)$. \\
\textbf{Case 2.} If $\lambda^*_u \geq \frac{1}{2\tau_{-1}b}$, i.e., if $s \geq \frac{(2\tau_{-1}\nu)^2}{2\tau_{-1}b}$, then the \emph{constrained} minimizer of $h_\Acal$ is $\lambda^* = \frac{1}{2\tau_{-1}b}$ with $h_\Acal(\lambda^*) = \exp\(-\frac{st}{2\tau_{-1}b} + \frac{t(\tau_{-1}\nu)^2}{2(\tau_{-1}b)^2}\)$. Also note that since $s \geq \frac{(2\tau_{-1}\nu)^2}{2\tau_{-1}b} = \frac{2(\tau_{-1}\nu)^2}{\tau_{-1}b}$, one has $\frac{(\tau_{-1}\nu)^2}{\tau_{-1}b} \leq \frac{s}{2}$.  Combining these two results shows
\begin{align*}
h_\Acal(\lambda^*)
&= \exp\(-\frac{st}{2\tau_{-1}b} + \frac{t(\tau_{-1}\nu)^2}{2(\tau_{-1}b)^2}\)
= \exp\( -\frac{st}{2\tau_{-1}b} + \frac{t}{2\tau_{-1}b} \cdot \frac{(\tau_{-1}\nu)^2}{\tau_{-1}b} \) \\
&\leq \exp\( -\frac{st}{2\tau_{-1}b} + \frac{st}{4\tau_{-1}b}\) 
= \exp\( -\frac{st}{4\tau_{-1}b} \).
\end{align*}
Combining the above two cases gives 
\bequation\label{eq:bound-A-II}
  \begin{aligned}
\P[\Acal]
&\leq \max\left\{\exp\(-\frac{s^2t}{2(2\tau_{-1}\nu)^2}\), \exp\(-\frac{st}{4\tau_{-1}b}\)\right\} \\
&= \mathrm{exp}\(-\min\left\{\frac{s^2 t}{2(2\tau_{-1}\nu)^2},\frac{st}{2(2\tau_{-1}b)}\right\}\),
  \end{aligned}
  \eequation
which is our desired bound on $\P[\Acal]$.

To bound $\P[\Bcal]$, first, the law of total probability gives
\[
\begin{aligned}
  &\ \P[\Bcal] \\
  =&\ \P\left[
\underbrace{T_{\epsilon}>t, \, \,
\frac{1}{t}\sum_{k=0}^{t-1}
\omega(\epsilonbar_f,E_k+E_k^+)
\leq \omega(\epsilonbar_f,2\epsilonbar_f)+s, \ \text{and} \ 
\sum_{k=0}^{t-1}\Theta_k I_k L_k
<
\(\hat{p}-\frac{1}{2}\)t-\frac{(\ibar+1)}{2}}_{\Bcal_1}
\right] \\
  & + \P\left[
\underbrace{T_{\epsilon}>t, \, \, 
\frac{1}{t}\sum_{k=0}^{t-1}
\omega(\epsilonbar_f,E_k+E_k^+)
\leq \omega(\epsilonbar_f,2\epsilonbar_f)+s, \ \text{and} \ 
\sum_{k=0}^{t-1}\Theta_k I_k L_k
\geq
\(\hat{p}-\frac{1}{2}\)t-\frac{(\ibar+1)}{2}}_{\Bcal_2}
\right].
\end{aligned}
\]
It follows from the definition of $h$ and the fact that $L_k = 1$ implies $A_k \geq \alphabar$ that $L_k = 1$ also implies $h(A_k) \geq h(\alphabar)$. One can combine this fact with Lemma~\ref{lem:alg-properties}(ii)(a) and Lemma~\ref{lem:alg-properties}(ii)(c) to get, for all $k < T_{\epsilon}$, that
\begin{equation}\label{eq:two_case_reduction}
\begin{aligned}
W_{k+1} &\leq W_k - h(\bar\alpha) + \omega(\epsilonbar_f, E_k + E_k^+) \\
  &\leq W_k - h(\bar\alpha) + \omega(\epsilonbar_f, 2\epsilonbar_f) \ \ \text{if $L_kI_k\Theta_k = 1$, and} \\
  W_{k+1} &\leq W_k + \omega(\epsilonbar_f, E_k + E_k^+) \ \ \text{if $L_kI_k\Theta_k = 0$.}
\end{aligned}
\end{equation}
  One can now prove that $\P[\Bcal_2]=0$ by showing that 
\begin{equation}\label{contradiction-firstpart}
T_{\epsilon} > t
\ \ \text{and} \ \
\frac{1}{t}\sum_{k=0}^{t-1}
\omega(\epsilonbar_f, E_k+E_k^+)
\leq \omega(\epsilonbar_f,2\epsilonbar_f)+s
\end{equation}
together imply that
\[
\sum_{k=0}^{t-1}\Theta_k I_k L_k
<
\(\hat{p}-\frac{1}{2}\)t-\frac{(\ibar+1)}{2}.
\]
For a proof by contradiction, suppose that~\eqref{contradiction-firstpart} holds and $\sum_{k=0}^{t-1}\Theta_k I_k L_k
\geq
\(\hat{p}-\frac{1}{2}\)t-\frac{(\ibar+1)}{2}$. This inequality, \eqref{eq:two_case_reduction}, $h(\alphabar) > 0$, and the second inequality in~\eqref{contradiction-firstpart} show that
\begin{align*}
W_t
  =&\ W_0 + \sum_{k = 0}^{t-1} (W_{k+1} - W_k) \\
  =&\ W_0 + \sum_{k = 0}^{t-1} (W_{k+1} - W_k)1\{L_kI_k\Theta_k = 1\} + \sum_{k = 0}^{t-1}(W_{k+1} - W_k)1\{L_kI_k\Theta_k = 0\} \\
  \leq&\ W_0 + \sum_{k=0}^{t-1}\(\omega(\epsilonbar_f, E_k + E_k^+) - h(\bar\alpha)\)1\{L_kI_k\Theta_k = 1\} \\
  &\ + \sum_{k = 0}^{t-1}\omega(\epsilonbar_f,E_k+E_k^+) 1\{L_kI_k\Theta_k = 0\}  \\
  =&\ W_0 - \sum_{k=0}^{t-1}1\{L_kI_k\Theta_k = 1\} h(\bar\alpha) + \sum_{k = 0}^{t-1}\omega(\epsilonbar_f,E_k+E_k^+) \\
  \leq&\ W_0 - \Big[(\phat-\half)t-\tfrac{\ibar+1}{2}\Big]h(\bar\alpha) + t\(\omega(\epsilonbar_f,2\epsilonbar_f) + s\) \\
  =&\ W_0 - \(\(\hat{p} - \frac{1}{2}\)h(\bar\alpha) - \omega(\epsilonbar_f, 2\epsilonbar_f) - s\)t + \frac{(\ibar+1)}{2} h(\bar\alpha) \leq 0,
\end{align*}
where the last inequality follows from~\eqref{t-geq}.
This inequality and the definition of $W_t$ gives $W_t = 0$, which means that $X_t$ is a local minimizer to problem~\eqref{prob:main}; in particular, $c(x_t) = 0$ and $X_t\in\Omega$.  It now follows from~\eqref{eq:delta} that $\Delta_t = 0$ so that $V_t = V_t^c = 0$ (see~\eqref{subprob:trust-region_conditions}). Combining this fact with Assumption~\ref{ass:strong-licq} shows that $[\nabla c(X_t), I_{\Acal_t}^T]^T$ has full row rank, where $\Acal_t = \{i \in [n]: [X_t]_i = 0 \}$. Thus, the LICQ holds for problem~\eqref{prob:main} at $X_t$. Since the LICQ holds at the minimizer $X_t$, one concludes that  $U_t^{\text{true}} = 0$, which in turn implies $S_t^{\text{true}} = V_t + U_t^{\text{true}} = 0$. Combining this fact with the complementary slackness condition (see~\eqref{kkt:complementarity}) for subproblem~\eqref{subprob:tangential} with $k$ replaced by $t$ and $g_k$ replaced by the true gradient $\nabla f(X_t)$, one finds that $Z_t^{\text{true}}$ satisfies $\min\{X_t+V_t+U_t^{\text{true}},-Z_t^{\text{true}}\} = \min\{X_t,-Z_t^{\text{true}}\} = 0$. To summarize, one has $S_t^{\text{true}} = 0$, $c(X_t) = 0$, $V_t = 0$, and $\min\{X_t,-Z_t^{\text{true}}\} = 0$, which means that $T_\epsilon \leq t$, which contradicts the first inequality in~\eqref{contradiction-firstpart}.  Therefore, $\P[\Bcal_2] = 0$.

Next, to bound $\P[\Bcal_1]$, Lemma~\ref{lem:true-total-prob} and Lemma~\ref{lem:zero-prob} give
\begin{align}
  &\ \P[\Bcal_1] \nonumber \\
  \leq&\ \P{}\left[ T_{\epsilon} > t \ \ \text{and} \ \ \sum_{k=0}^{t-1} \Theta_k I_k L_k < \( \hat{p} - \frac{1}{2} \) t - \frac{(\ibar+1)}{2} \right] \nonumber \\
  =&\ \P{}\left[ T_{\epsilon} > t, \,\, \sum_{k=0}^{t-1} \Theta_k I_k L_k < \(\hat{p} - \frac{1}{2} \) t - \frac{(\ibar+1)}{2}, \ \ \text{and} \ \ \sum_{k=0}^{t-1} I_k < \hat{p} t \right] \nonumber \\
  &\ + \P{}\left[ T_{\epsilon} > t, \, \, \sum_{k=0}^{t-1} \Theta_k I_k L_k < \(\hat{p} - \frac{1}{2} \) t - \frac{(\ibar+1)}{2}, \ \ \text{and} \ \ \sum_{k=0}^{t-1} I_k \geq \hat{p}t \right] \nonumber \\
  \leq&\ \P{}\left[ \sum_{k=0}^{t-1} I_k < \hat{p}t \right] + \P{}\left[ T_{\epsilon} > t, \,\, \sum_{k=0}^{t-1} \Theta_k I_k L_k < \( \hat{p} - \frac{1}{2} \) t - \frac{(\ibar+1)}{2}, \ \ \text{and} \ \ \sum_{k=0}^{t-1} I_k \geq \hat{p} t \right] \nonumber \\
  \leq&\ \mathrm{exp}\(-\frac{(p - \hat{p})^2}{2p^2} t\) + 0 = \mathrm{exp}\(-\frac{(p - \hat{p})^2}{2p^2} t\). \label{bound-probB1}
\end{align}

Finally, one can use $\P{}[\Bcal_2] = 0$ to obtain 
\begin{align*}
\P[T_\epsilon \leq t]
&= 1 - \P[T_\epsilon > t]
= 1 - \P[\Acal] - \P[\Bcal]
= 1 - \P[\Acal] - \P[\Bcal_1].
\end{align*}
The result~\eqref{eq:subexponential} follows from this equation, \eqref{eq:bound-A-II}, and~\eqref{bound-probB1}. \qed
\end{proof}

The parameter $s$ measures how much ``above expectation" the function evaluation error is allowed to be. From  Lemma~\ref{lem:alg-properties}(ii)(c), we have $W_{k+1} - W_k \leq \omega(\epsilonbar_f, E_k + E_k^+) = 2\epsilonbar_f + \tau_{-1}(E_k + E_k^+)$. The parameter $s$ serves as a bound on how far $\tau_{-1}(E_k + E_k^+)$ exceeds its mean bound $2\tau_{-1}\epsilonbar_f$; a Bernstein-type inequality is then used to control the probability that the cumulative excess exceeds its mean bound over $t$ iterations, yielding the second exponential in~\eqref{eq:subexponential}. Thus, $s$ trades off two effects: a smaller $s$ slows the decay of $\exp(-\min\{s^2 t/(2(2\tau_{-1}\nu)^2), st/(2(2\tau_{-1}b))\})$ but enlarges the admissible range for $\hat{p}$ and lowers the threshold on $t$; and a larger $s$ sharpens the concentration bound but shrinks the range for $\hat{p}$ (vanishing as $s \to (p - \tfrac{1}{2})h(\bar{\alpha}) - \omega(\epsilonbar_f, 2\epsilonbar_f)$) and raises the threshold on $t$. If $s = 0$ and we further assume, for all $k \in \N{}$, that $E_k + E_k^+ \leq 2\epsilon_f$ then the Bernstein-type term in~\eqref{eq:subexponential} is no longer needed, and the concentration argument invoked on $E_k + E_k^+$ is also not needed.  Therefore, in this case, the bound in~\eqref{eq:subexponential}  simplifies to 
\begin{equation*}
\P{}(T_{\epsilon} \leq t) \ge 
1 - \exp\(
-\frac{(p - \hat{p})^2}{2p^2}  t
\),
\end{equation*}
which is a high-probability iteration complexity in the bounded-noise setting.


From Theorem~\ref{thm:main}, we can derive an almost sure termination result and an  expected iteration complexity result, both of which are stated in the next result.

\begin{corollary}\label{cor:convergence_expected_complexity}
Let the suppositions in Theorem~\ref{thm:main} hold and suppose that $\epsilon$ satisfies~\eqref{eq:epsilon_lower_bound}.   
\begin{itemize}
\item[(i)] 
An $\epsilon$-KKT point is computed in a finite number of iterations with probability~$1$.
\item[(ii)] 
The expected number of iterations to reach an $\epsilon$-KKT point is $O(\epsilon^{-2})$.
\end{itemize}
\end{corollary}
\begin{proof}
We first prove part (i).  Fix any $s$ and $\hat{p}$ satisfying~\eqref{s-in} and \eqref{phat-in}, respectively, and let $\that$ be as defined in~\eqref{t-geq}. 
With $A_t := \{T_{\epsilon}> t\}$, Theorem~\ref{thm:main} implies 
\begin{align*}
& \ \sum_{t=\lceil \that\, \rceil}^{\infty} \P{}(A_t) \\
\leq & \ \sum_{t=\lceil \that\, \rceil}^{\infty} \left[\exp\(
-\frac{(p - \hat{p})^2}{2p^2} \, t
\)
+ \exp\(
-\min\left\{
\frac{s^2 t}{2(2\tau_{-1}\nu)^2}, \,
\frac{s t}{2(2\tau_{-1}b)}
\right\}
\)\right] \\
= & \ \sum_{t=\lceil \that\, \rceil}^{\infty} \left\{ {\left[\underbrace{\exp\(
-\frac{(p - \hat{p})^2}{2p^2}
\)}_{<1}\right]}^t
+ {\left[\underbrace{\exp\left(
-\min\left\{
\frac{s^2}{2(2\tau_{-1}\nu)^2}, \,
\frac{s}{2(2\tau_{-1}b)}
\right\}\right)}_{< 1}\right]}^t \right\} \\
< & \ \infty,
\end{align*}
which uses $\hat{p} < p$ and $s > 0$. It follows from this and the Borel-Cantelli lemma that $\P[\lim\sup_{t \to \infty} A_t] = 0$. Thus, $\P[T_\epsilon < \infty] = 1$, meaning that an $\epsilon$-KKT point is found in a finite number of iterations with probability~1.

Next, consider part (ii).  Let us define the following:
\begin{alignat}{2}
\rho &:= \tfrac{\omega(\epsilonbar_f,2\epsilonbar_f)}{h(\alphabar)}, &
\sigma &:= (p-\thalf) - \rho, \label{def:rho-and-sigma} \\ 
M &:= \max\{\sqrt{2},\kappa_T\}, \qquad & 
\chi &:= \min\left\{\tfrac{1}{2(1+\kappa)^2},\tfrac{1-2\eta}{2}\right\}, \label{def:M-and-chi}
\end{alignat}
and note that it follows from~\eqref{def:M-and-chi} that $h(\alphabar)$ from Lemma~\ref{lem:alg-properties} may be written as 
\begin{equation}\label{def:h-equivalent}
h(\alphabar) \equiv \eta_1M^{-2}\chi\taumin\alphabar\epsilon^2.
\end{equation}
It follows from~\eqref{def:h-equivalent}, \eqref{eq:epsilon_lower_bound}, and $\max\left\{(1+\kappa)^2,\frac{1}{1-2\eta}\right\} = \tfrac{1}{2\chi}$ that
\begin{align*}
\frac{M^2h(\alphabar)}{\eta_1\chi\taumin\alphabar}
&= \epsilon^2
> M^2\frac{1}{2\chi}\frac{(4+4\tau_{-1})\epsilonbar_f}{\theta\taumin\alphabar\eta_1(p-1/2)} 
= M^2\frac{1}{\chi}\frac{(2+2\tau_{-1})\epsilonbar_f}{\theta\taumin\alphabar\eta_1(p-1/2)},
\end{align*}
which after simplification and using $\omega(\epsilonbar_f,2\epsilonbar_f) = (2+2\tau_{-1})\epsilonbar_f$, one obtains
\begin{equation}\label{rho-bound-above}
\rho = \frac{\omega(\epsilonbar_f,2\epsilonbar_f)}{h(\alphabar)} < \theta(p-\thalf)
\ \implies \ \sigma > (1-\theta)(p-\thalf) > 0.
\end{equation}
Next, let us choose for $s$, which is required to satisfy~\eqref{s-in}, the specific value
\begin{equation}\label{def:specific-s}
s = \frac{\sigma h(\alphabar)}{3} = \frac{(p-\thalf - \rho)h(\alphabar)}{3} = \frac{(p-\thalf)h(\alphabar) - \omega(\epsilonbar_f,2\epsilonbar_f)}{3},
\end{equation}
which satisfies the requirement in~\eqref{s-in}. With this choice of $s$, let us next choose for $\phat$, which is required to satisfy~\eqref{phat-in}, the specific value
\begin{align*}
\phat &= \half + \rho + \frac{2\sigma}{3}.
\end{align*}
Note that the lower bound in~\eqref{phat-in}, with $s$ chosen as above, satisfies
$$
\half +\frac{\omega(\epsilonbar_f,2\epsilonbar_f)+s}{h(\alphabar)}
= \half +\frac{\omega(\epsilonbar_f,2\epsilonbar_f)+\sigma h(\alphabar)/3}{h(\alphabar)}
= \half + \rho + \frac{\sigma}{3} < \phat,
$$
meaning that $\phat$ as chosen above satisfies the lower bound in~\eqref{phat-in}. On the other hand, one has from the definition of $\sigma$, \eqref{rho-bound-above}, and $\theta\in(0,1)$ that
$$
\phat 
= \half + \rho + \frac{2\sigma}{3}
= \half + \frac{2}{3}\left(p-\half\right) + \frac{1}{3} \rho
< \half + \left(\frac{2}{3} + \frac{1}{3}\theta\right)\left(p-\half\right)
< \half + \left(p-\half\right)
= p
$$
meaning that $\phat$ also satisfies the upper bound in~\eqref{phat-in}. With these specific choices for $s$ and $\phat$, one can now compute the denominator of $\that$ in~\eqref{t-geq} as being equal to
\begin{equation}\label{def:Delta-equiv}
\begin{aligned}
\Delta 
&:= \hat{p} - \left( \frac{1}{2} + \tfrac{\omega(\bar{\epsilon}_f, 2\bar{\epsilon}_f) + s}{h(\bar{\alpha})}\right) 
= \half + \rho + \frac{2\sigma}{3} - \left( \frac{1}{2} + \tfrac{\omega(\bar{\epsilon}_f, 2\bar{\epsilon}_f) + \frac{\sigma h(\alphabar)}{3}}{h(\bar{\alpha})}\right) \\
&=  \half + \rho + \frac{2\sigma}{3} - \left( \frac{1}{2} + \rho + \frac{\sigma}{3}\right)
= \frac{\sigma}{3},
\end{aligned}
\end{equation}
and can similarly calculate, for the choice of $\phat$ above, that
\begin{equation}\label{diff-ps}
p-\phat = p - \half - \rho - \frac{2\sigma}{3} = \sigma -  \frac{2\sigma}{3} = \frac{\sigma}{3}.
\end{equation}
One can now also see from~\eqref{t-geq}, \eqref{def:Delta-equiv}, \eqref{def:h-equivalent}, \eqref{rho-bound-above}, and~\eqref{def:M-and-chi} that
\begin{equation}\label{that-bigO-bound}
\begin{aligned}
\that 
&= \frac{\Wbar_0/h(\alphabar) + (\ibar+1)/2}{\Delta} 
= \frac{3}{\sigma}
\left( \frac{\Wbar_0}{\eta_1M^{-2}\chi\taumin\alphabar\epsilon^2} + \frac{\ibar+1}{2} \right) \\
&< \frac{3}{(1-\theta)(p-\thalf)}
\left( \frac{M^2\Wbar_0}{\eta_1\chi\taumin\alphabar\epsilon^2} + \frac{\ibar+1}{2} \right) \\
&= \frac{3}{(1-\theta)(p-\thalf)}
\left( \frac{\max\{2,\kappa_T^2\}\Wbar_0}{\eta_1\min\left\{\tfrac{1}{2(1+\kappa)^2},\tfrac{1-2\eta}{2}\right\}\taumin\alphabar\epsilon^2} + \frac{\ibar+1}{2} \right)
= O(\epsilon^{-2}).
\end{aligned}
\end{equation}
Next, it follows from Theorem~\ref{thm:main} that, for all $t \geq \that$,
\begin{equation}\label{tail-prob-bound-for-corollary}
\P[T_\epsilon > t] \leq \exp\(-\frac{(p - \hat{p})^2 t}{2p^2}\) + \exp\(-\min\left\{\frac{s^2 t}{2(2\tau_{-1}\nu)^2}, \frac{st}{2(2\tau_{-1}b)}\right\}\).
\end{equation}
Using the tail sum formula for expectation, one has that $\E[T_{\epsilon}] = \sum_{k=0}^{\infty} \P[T_{\epsilon} > k]$, which one can split at the ceiling of $\that$ to obtain
\begin{equation}\label{eq:expectation.stopping.time}
\E[T_{\epsilon}]
= \sum_{k=0}^{\lceil \that\, \rceil} \P[T_{\epsilon} > k] + \sum_{m=1}^\infty \P[T_{\epsilon} > \lceil \that\, \rceil + m]
\leq \lceil \that\, \rceil + 1 + \sum_{m=1}^{\infty} \P[T_{\epsilon} > \lceil \that\, \rceil + m].
\end{equation}
Next, let us define $c_1 := \tfrac{(p - \hat{p})^2}{2p^2} > 0$ and $c_2 := \min\left\{\tfrac{s^2}{2(2\tau_{-1}\nu)^2}, \tfrac{s}{2(2\tau_{-1}b)}\right\} > 0$, which are motivated by~\eqref{tail-prob-bound-for-corollary}.  One can use~\eqref{diff-ps} and~\eqref{rho-bound-above} to obtain
\begin{equation}\label{c1-bound}
c_1 
= \frac{(p-\phat)^2}{2p^2} 
= \frac{(\sigma/3)^2}{2p^2} 
= \frac{\sigma^2}{18p^2} 
> \frac{(1-\theta)^2(p-\thalf)^2}{18p^2} 
=: \underline{c_1}.
\end{equation}
To bound $c_2$, we must compute appropriate bounds for $s$, $\nu$, and $b$.  To bound $s$, it follows from~\eqref{def:specific-s}, \eqref{def:rho-and-sigma}, and~\eqref{rho-bound-above} that
\begin{equation}\label{s-bound-above}
\begin{aligned}
s 
&= \frac{\sigma h(\alphabar)}{3}
= \frac{\sigma}{3}\frac{\omega(\epsilonbar_f,2\epsilonbar_f)}{\rho} \\
&\geq \frac{(1-\theta)(p-\thalf)}{3}\frac{(2+2\tau_{-1})\epsilonbar_f}{\theta(p-\thalf)}
= \frac{(1-\theta)}{3}\frac{(2+2\tau_{-1})\epsilonbar_f}{\theta}.
\end{aligned}
\end{equation}
Next, since $p > 1/2$, we know that $1/2 + p > 1-\delta$, which implies that $1/2 > 1-\delta-p$, which in turn implies that $2 < 1/(1-\delta-p)$, which in turn implies that
\begin{equation}\label{L-bound-below}
L:= \ln(1/(1-\delta-p)) > \ln(2). 
\end{equation}
We can now also obtain from how $p$ is defined in~\eqref{def:p} that
\begin{equation}\label{L-is-min}
L = \ln\left(\frac{1}{1-\delta-p}\right)
\equiv \min\left\{\frac{u_f^2}{2\nu^2}, \frac{u_f}{2b}\right\}. 
\end{equation}
It follows from~\eqref{L-is-min} and the definition of $u_f$ in Assumption~\ref{ass:p} that
$$
L \leq \frac{u_f^2}{2\nu^2} < \frac{\epsilonbar_f^2}{2\nu^2}
\ \ \text{and} \ \ 
L \leq \frac{u_f}{2b} < \frac{\epsilonbar_f}{2b},
$$
which after rearranging and using~\eqref{L-bound-below} gives
$$
\nu < \frac{\epsilonbar_f}{\sqrt{2L}} < \frac{\epsilonbar_f}{\sqrt{2\ln(2)}} =  \frac{\epsilonbar_f}{\sqrt{\ln(4)}}
\ \ \text{and} \ \ 
b < \frac{\epsilonbar_f}{2L} < \frac{\epsilonbar_f}{2\ln(2)} = \frac{\epsilonbar_f}{\ln(4)} .
$$
Using these bounds on $\nu$ and $b$, and the bound on $s$ in~\eqref{s-bound-above}, one has
\begin{align*}
\frac{s}{\nu} &\geq \frac{\sqrt{\ln(4)}(1-\theta)(2+2\tau_{-1})}{3\theta} \ \ \text{and} \ \ 
\frac{s}{b} \geq \frac{\ln(4)(1-\theta)(2+2\tau_{-1})}{3\theta}.
\end{align*}
These bounds may now be combined with the definition of $c_2$ to obtain
\begin{equation}\label{c2-bound}
\begin{aligned}
c_2 
&= \min\left\{\tfrac{s^2}{2(2\tau_{-1}\nu)^2}, \tfrac{s}{2(2\tau_{-1}b)}\right\} \\
&\geq \min\left\{\tfrac{\ln(4)(1-\theta)^2(1+\tau_{-1})^2}{18\tau_{-1}^2\theta^2}, \tfrac{\ln(4)(1-\theta)(1+\tau_{-1})}{6\tau_{-1}\theta}\right\}
=: \underline{c_2}.
\end{aligned}
\end{equation}
Since the exponents in the bound~\eqref{tail-prob-bound-for-corollary} are linear in $t$, one has 
\begin{align*}
  &\ \sum_{m=1}^{\infty} \P[T_{\epsilon} > \lceil \that\, \rceil + m] \\
  \leq&\ \sum_{m=1}^{\infty} \mathrm{exp}(-c_1(\lceil \that\, \rceil + m)) + \sum_{m=1}^{\infty} \mathrm{exp}(-c_2(\lceil \that\, \rceil + m)) \\
  \leq&\ \sum_{m=1}^{\infty} \mathrm{exp}(-c_1 m) + \sum_{m=1}^{\infty} \mathrm{exp}(-c_2 m) = \frac{\mathrm{exp}(-c_1)}{1 - \mathrm{exp}(-c_1)} + \frac{\mathrm{exp}(-c_2)}{1 - \mathrm{exp}(-c_2)} \\
  =& \frac{1}{ \mathrm{exp}(c_1)-1} + \frac{1}{ \mathrm{exp}(c_2)-1}
  \leq \frac{1}{ \mathrm{exp}(\underline{c_1})-1} + \frac{1}{ \mathrm{exp}(\underline{c_2})-1},
\end{align*}
where one has that this upper bound is independent of $\epsilon$ since both $\underline{c_1}$ and $\underline{c_2}$ are independent of $\epsilon$. One can now combine this result with~\eqref{eq:expectation.stopping.time} and~\eqref{that-bigO-bound} to obtain
\begin{align*}
\E[T_{\epsilon}]
&\leq  \lceil \that\, \rceil + 1 + \frac{1}{ \mathrm{exp}(\underline{c_1})-1} + \frac{1}{ \mathrm{exp}(\underline{c_2})-1}
= O(\epsilon^{-2}).
\end{align*}
The desired result follows by combining this inequality with Lemma~\ref{lem:kkt-result}. \qed
\end{proof}

\section{Conclusion}\label{sec.conclusion}
In this paper, we analyzed a proximal-gradient method for solving regularized stochastic optimization problems with general constraints, where exact objective function and gradient evaluations are inaccessible. We established a high-probability complexity bound under assumptions that accommodate noisy tail distributions commonly encountered in machine learning applications. Our analysis yields three main results: (i) explicit high-probability iteration complexity bounds for both subexponential and bounded-noise settings (Theorem~\ref{thm:main} and the discussion following that theorem), (ii) almost-sure termination guarantee that an $\epsilon$-KKT point is found in a finite number of iterations (Corollary~\ref{cor:convergence_expected_complexity}(i)), and (iii) an expected iteration complexity of $O(\epsilon^{-2})$ (Corollary~\ref{cor:convergence_expected_complexity}(ii)). These results extend the high-probability analysis framework found in~\cite{jin2024high} from unconstrained smooth problems to the setting of constrained nonsmooth problems.
\section{Probabilistic tools}



This section contains some probabilistic results  needed to handle the stochastic gradient estimates. We first provide standard closure properties of one-sided subexponential random variables under affine transformations and summation.


\begin{lemma}\label{lem:subexponential-properties}
Let $X$ and $Y$ be one-sided subexponential random variables
with strictly positive parameters $(\nu_1,b_1)$ and $(\nu_2,b_2)$, respectively; i.e.,
\begin{align*}
\mathbb{E}\big[e^{\lambda(X-\mathbb{E}[X])}\big]
&\le e^{\lambda^2\nu_1^2/2}
\ \ \text{for all } \lambda\in[0,1/b_1] \ \text{and} \\
\mathbb{E}\big[e^{\lambda(Y-\mathbb{E}[Y])}\big]
&\le e^{\lambda^2\nu_2^2/2}
\ \ \text{for all } \lambda\in[0,1/b_2].
\end{align*}
Then, the following results hold:
\begin{enumerate}
\item[(i)] For any $a>0$ and $c\in\R{}$, $aX+c$ is one-sided subexponential with parameters $(a\nu_1,\,a b_1)$.
\item[(ii)] $X+Y$ is one-sided subexponential with parameters $\big(\sqrt{2(\nu_1^2+\nu_2^2)},\,\max\{2b_1,2b_2\}\big)$.
\end{enumerate}
\end{lemma}

\begin{proof}
For part (i), we bound the moment-generating function of $aX+c$ as
\[
\mathbb{E}\big[e^{\lambda(aX+c-\mathbb{E}[aX+c])}\big]
= \mathbb{E}\big[e^{a\lambda(X-\mathbb{E}[X])}\big]
\le e^{a^2\lambda^2\nu_1^2/2}
\quad\text{for all } \lambda\in\Big[0,\frac{1}{a b_1}\Big],
\]
where the above inequality uses the fact that $X$ is one-sided
subexponential with parameters $(\nu_1,b_1)$. This proves part (i).

To prove part (ii), let
$\lambda\in\left[0,\min\left\{\tfrac{1}{2b_1},\tfrac{1}{2b_2}\right\}\right] = [0,\tfrac{1}{\max\{2b_1,2b_2\}}]$ so that
$2\lambda\in[0,1/b_1]\cap[0,1/b_2]$. The Cauchy-Schwarz inequality for expectations gives that
\begin{align*}
& \ \mathbb{E}\big[e^{\lambda(X+Y-\mathbb{E}[X+Y])}\big] \\
\leq & \ \Big(\mathbb{E}\big[e^{2\lambda(X-\mathbb{E}[X])}\big]\,
\mathbb{E}\big[e^{2\lambda(Y-\mathbb{E}[Y])}\big]\Big)^{1/2} \\
\leq & \ \Big(e^{4\lambda^2\nu_1^2/2}\,e^{4\lambda^2\nu_2^2/2}\Big)^{1/2}
= e^{\lambda^2(\nu_1^2+\nu_2^2)} \ \text{for all} \ \lambda \in \left[0,\tfrac{1}{\max\{2b_1,2b_2\}}\right], 
\end{align*}
which completes the proof of part (ii). \qed 
\end{proof}



The next result is a one-sided Bernstein's inequality.

\begin{lemma}\label{lem:bernstein}
If $X$ is a random variable satisfying, with strictly positive parameters $(\nu,b)$, 
\begin{equation}\label{app:claim-bernstein}
\E\left[e^{\lambda(X - \E[X])}\right] \leq e^{\frac{\lambda^2\nu^2}{2}}
\quad \text{for all } \lambda \in \left[0, \frac{1}{b}\right],
\end{equation}
then, for any $t \geq 0$, it holds that
\begin{equation}\label{app:1sided-bernstein}
\P(X - \E[X] \geq t)
\leq \exp\left(-\min\left\{\frac{t^2}{2\nu^2}, \frac{t}{2b}\right\}\right).
\end{equation}
\end{lemma}
\begin{proof}
If $t = 0$, the right-hand side of~\eqref{app:1sided-bernstein} is equal to $1$ so that the bound holds trivially. Therefore, for the remainder of the proof we suppose that  $t > 0$. By combining Markov's inequality and~\eqref{app:claim-bernstein} we find, for all $\lambda\in[0,1/b]$, that
\begin{equation*}
\P(X - \E[X] \geq t)
\leq \P\left(e^{\lambda(X - \E[X])} \geq e^{\lambda t}\right)
\leq e^{-\lambda t}\,\E\!\left[e^{\lambda(X - \E[X])}\right]
\leq \exp\!\left(\frac{\lambda^2\nu^2}{2} - \lambda t\right).
\end{equation*}
To obtain the tightest bound, we minimize
$h_b(\lambda) = \frac{\lambda^2\nu^2}{2} - \lambda t$ over $\lambda \in [0, 1/b]$.  
Since $\nu^2 > 0$, $h_b$ is a strongly convex quadratic function with \emph{unconstrained} minimizer
at $\lambda_u = t/\nu^2 > 0$, which follows from solving $h_b'(\lambda) = \lambda\nu^2 - t = 0$ for $\lambda$. We now consider two cases depending on whether $\lambda_u$ satisfies the constraint $\lambda \leq 1/b$. If $t \leq \nu^2/b$, then $\lambda^* = \lambda_u \in (0, 1/b]$ so that $h_b(\lambda^*) = -t^2/(2\nu^2)$.
On the other hand, if $t > \nu^2/b$, then $\lambda^* = 1/b$.  Since 
$\nu^2/b < t$ implies $\nu^2/(2b^2) < t/(2b)$, it follows that $h_b(\lambda^*) = \tfrac{\nu^2}{2b^2} - \tfrac{t}{b} \leq \tfrac{t}{2b} - \tfrac{t}{b} = -\tfrac{t}{2b}$. Taking both of these cases into account, we find for all $\lambda\in[0,1/b]$ that
\begin{align*}
\P(X - \E[X] \geq t)
&\leq 
\exp\left(\max\{-\tfrac{t^2}{2\nu^2}, -\tfrac{t}{2b}\}\right)
=  \exp\left(-\min\{\tfrac{t^2}{2\nu^2}, \tfrac{t}{2b}\}\right),
\end{align*}
which completes the proof. \qed
\end{proof}

\bibliographystyle{plain}
\bibliography{dpr_fec}

@article{azuma1967weighted,
  title={Weighted sums of certain dependent random variables},
  author={Azuma, Kazuoki},
  journal={Tohoku Mathematical Journal, Second Series},
  volume={19},
  number={3},
  pages={357--367},
  year={1967},
  publisher={Mathematical Institute, Tohoku University}
}

@article{jin2024high,
  title={High probability complexity bounds for adaptive step search based on stochastic oracles},
  author={Jin, Billy and Scheinberg, Katya and Xie, Miaolan},
  journal={SIAM Journal on Optimization},
  volume={34},
  number={3},
  pages={2411--2439},
  year={2024},
  publisher={SIAM}
}

@inproceedings{chen2018constraint,
  title={Constraint-aware deep neural network compression},
  author={Chen, Changan and Tung, Frederick and Vedula, Naveen and Mori, Greg},
  booktitle={Proceedings of the European Conference on Computer Vision (ECCV)},
  pages={400--415},
  year={2018}
}

@article{shi2025momentum,
  title={A momentum-based linearized augmented Lagrangian method for nonconvex constrained stochastic optimization},
  author={Shi, Qiankun and Wang, Xiao and Wang, Hao},
  journal={Mathematics of Operations Research},
  year={2025},
  publisher={INFORMS}
}

@article{boob2023stochastic,
  title={Stochastic first-order methods for convex and nonconvex functional constrained optimization},
  author={Boob, Digvijay and Deng, Qi and Lan, Guanghui},
  journal={Mathematical Programming},
  volume={197},
  number={1},
  pages={215--279},
  year={2023},
  publisher={Springer}
}

@article{boob2025level,
  title={Level constrained first order methods for function constrained optimization},
  author={Boob, Digvijay and Deng, Qi and Lan, Guanghui},
  journal={Mathematical Programming},
  volume={209},
  number={1},
  pages={1--61},
  year={2025},
  publisher={Springer}
}

@article{jin2022stochastic,
  title={A stochastic primal-dual method for a class of nonconvex constrained optimization},
  author={Jin, Lingzi and Wang, Xiao},
  journal={Computational Optimization and Applications},
  volume={83},
  number={1},
  pages={143--180},
  year={2022},
  publisher={Springer}
}

@article{idrees2025constrained,
  title={Constrained stochastic recursive momentum successive convex approximation},
  author={Idrees, Basil M and Arora, Lavish and Rajawat, Ketan},
  journal={IEEE Transactions on Signal Processing},
  year={2025},
  publisher={IEEE}
}

@article{gorbunov2024high,
  title={High-Probability Complexity Bounds for Non-smooth Stochastic Convex Optimization with Heavy-Tailed Noise},
  author={Gorbunov, Eduard and Danilova, Marina and Shibaev, Innokentiy and Dvurechensky, Pavel and Gasnikov, Alexander},
  journal={Journal of Optimization Theory and Applications},
  volume={203},
  number={3},
  pages={2679--2738},
  year={2024},
  publisher={Springer}
}

@inproceedings{li2020highpro,
    author={Li, Xiaoyu and Orabona, Francesco},
    title={A High Probability Analysis of Adaptive {SGD} with Momentum},
    booktitle={Workshop on Beyond First Order Methods in ML Systems at the 37th International Conference on Machine Learning (ICML)},
    year={2020},
    url={https://arxiv.org/abs/2007.14294}
}

@inproceedings{nguyen2018sgd,
  title={{SGD} and {Hogwild!} Convergence Without the Bounded Gradients Assumption},
  author={Nguyen, Lam M. and Nguyen, Phuong Ha and van Dijk, Marten and Richt{\'a}rik, Peter and Scheinberg, Katya and Tak{\'a}{\v{c}}, Martin},
  booktitle={Proceedings of the 35th International Conference on Machine Learning},
  pages={3750--3758},
  year={2018},
  volume={80},
  series={Proceedings of Machine Learning Research},
  publisher={PMLR}
}

@article{dai2025proximal,
  title={A Proximal-Gradient Method for Equality Constrained Optimization},
  author={Dai, Yutong and Qu, Xiaoyi and Robinson, Daniel P},
  journal={SIAM Journal on Optimization},
  volume={35},
  number={4},
  pages={2654--2683},
  year={2025},
  publisher={SIAM}
}

@inproceedings{sadiev2023high,
  title={High-probability bounds for stochastic optimization and variational inequalities: the case of unbounded variance},
  author={Sadiev, Abdurakhmon and Danilova, Marina and Gorbunov, Eduard and Horv{\'a}th, Samuel and Gidel, Gauthier and Dvurechensky, Pavel and Gasnikov, Alexander and Richt{\'a}rik, Peter},
  booktitle={International Conference on Machine Learning},
  pages={29563--29648},
  year={2023},
  organization={PMLR}
}

@article{cutkosky2021high,
  title={High-probability bounds for non-convex stochastic optimization with heavy tails},
  author={Cutkosky, Ashok and Mehta, Harsh},
  journal={Advances in Neural Information Processing Systems},
  volume={34},
  pages={4883--4895},
  year={2021}
}

@inproceedings{jain2019making,
  title={Making the last iterate of {SGD} information theoretically optimal},
  author={Jain, Prateek and Nagaraj, Dheeraj and Netrapalli, Praneeth},
  booktitle={Conference on Learning Theory},
  pages={1752--1755},
  year={2019},
  organization={PMLR}
}

@book{rockafellar1970convex,
  author    = {Rockafellar, R. Tyrrell},
  title     = {Convex Analysis},
  publisher = {Princeton University Press},
  address   = {Princeton, NJ},
  year      = {1970},
  series    = {Princeton Mathematical Series},
  volume    = {28}
}

@article{curtis2025proximal,
  title={A Proximal-Gradient Method for Solving Regularized Optimization Problems with General Constraints},
  author={Curtis, Frank E and Qu, Xiaoyi and Robinson, Daniel P},
  journal={arXiv preprint arXiv:2512.23166},
  year={2025}
}

@book{shapiro2021lectures,
  title={Lectures on stochastic programming: modeling and theory},
  author={Shapiro, Alexander and Dentcheva, Darinka and Ruszczynski, Andrzej},
  year={2021},
  publisher={SIAM}
}

@inproceedings{tomar2014manifold,
  title={Manifold regularized deep neural networks.},
  author={Tomar, Vikrant Singh and Rose, Richard C},
  booktitle={INTERSPEECH},
  pages={348--352},
  year={2014}
}

@article{zhu2019physics,
  title={Physics-constrained deep learning for high-dimensional surrogate modeling and uncertainty quantification without labeled data},
  author={Zhu, Yinhao and Zabaras, Nicholas and Koutsourelakis, Phaedon-Stelios and Perdikaris, Paris},
  journal={Journal of computational physics},
  volume={394},
  pages={56--81},
  year={2019},
  publisher={Elsevier}
}

@article{bertsimas2018robust,
  title={Robust sample average approximation},
  author={Bertsimas, Dimitris and Gupta, Vishal and Kallus, Nathan},
  journal={Mathematical Programming},
  volume={171},
  number={1},
  pages={217--282},
  year={2018},
  publisher={Springer}
}

@inproceedings{liu2024high,
  title={High-probability bound for non-smooth non-convex stochastic optimization with heavy tails},
  author={Liu, Langqi and Wang, Yibo and Zhang, Lijun},
  booktitle={Forty-first International Conference on Machine Learning},
  year={2024}
}

@article{berahas2025sequential,
  title={A sequential quadratic programming method with high-probability complexity bounds for nonlinear equality-constrained stochastic optimization},
  author={Berahas, Albert S and Xie, Miaolan and Zhou, Baoyu},
  journal={SIAM Journal on Optimization},
  volume={35},
  number={1},
  pages={240--269},
  year={2025},
  publisher={SIAM}
}

@inproceedings{mishchenko2020random,
  title={Random Reshuffling: Simple Analysis with Vast Improvements},
  author={Mishchenko, Konstantin and Khaled, Ahmed and Richt{\'a}rik, Peter},
  booktitle={Advances in Neural Information Processing Systems},
  volume={33},
  year={2020}
}

@article{gurbuzbalaban2021random,
  title={Why Random Reshuffling Beats Stochastic Gradient Descent},
  author={G{\"u}rb{\"u}zbalaban, Mert and Ozdaglar, Asuman and Parrilo, Pablo A.},
  journal={Mathematical Programming},
  volume={186},
  pages={49--84},
  year={2021}
}

\end{document}